\documentclass[12pt]{article}
\usepackage{amsfonts}
\usepackage{amssymb}
\usepackage{mathrsfs}
\usepackage{srcltx}
\usepackage{xcolor}
\usepackage{amsmath}
\usepackage{amsthm}
\usepackage{amstext}
\usepackage{amsopn}

\usepackage{graphicx}
\usepackage{bm}
\newtheorem{theorem}{Theorem}[section]
\newtheorem{lemma}[theorem]{Lemma}

\newtheorem{corollary}[theorem]{Corollary}

\theoremstyle{definition}

\theoremstyle{remark}
\newtheorem{remark}[theorem]{Remark}

\numberwithin{equation}{section}

\newcommand{\ba}{\begin{array}}
\newcommand{\ea}{\end{array}}
\newcommand{\f}{\frac}

\newcommand{\ep}{\varepsilon}
\newcommand{\la}{\lambda}
\newcommand{\R}{{\mathbb R}}
\newcommand{\ds}{\displaystyle}

\begin{document}
\date{}
\title{ \bf\large{Spectral Monotonicity of Perturbed Quasi-positive Matrices with Applications in Population Dynamics}\thanks{S. Chen is supported by National Natural Science Foundation of China (No 11771109) and a grant from China Scholarship Council,  J. Shi is supported by US-NSF grants DMS-1715651 and DMS-1853598, and Z. Shuai is supported by US-NSF grant DMS-1716445. 
}}
\author{
Shanshan Chen\footnote{Email: chenss@hit.edu.cn}\\[-1mm]
{\small Department of Mathematics, Harbin Institute of Technology}\\[-2mm]
{\small Weihai, Shandong, 264209, P. R. China}\\[5mm]
Junping Shi\footnote{Corresponding Author, Email: jxshix@wm.edu}\\[-1mm]
{\small Department of Mathematics, William \& Mary}\\[-2mm]
{\small Williamsburg, Virginia, 23187-8795, USA}\\[5mm]
Zhisheng Shuai\footnote{Email: shuai@ucf.edu}\\[-1mm]
{\small Department of Mathematics, University of Central Florida}\\[-2mm]
{\small Orlando, Florida, 32816, USA}\\[5mm]
Yixiang Wu\footnote{Email: yixiang.wu@mtsu.edu} \\[-1mm]
{\small Department of Mathematics, Middle Tennessee State University}\\[-2mm]
{\small Murfreesboro, Tennessee, 37132, USA}
}

\maketitle

\newpage

\begin{abstract}
Threshold values in population dynamics can be formulated as spectral bounds of matrices, determining the dichotomy of population persistence and extinction. For a square matrix $\mu A + Q$, where $A$ is a quasi-positive matrix describing population dispersal among patches in a heterogeneous environment and $Q$ is a diagonal matrix encoding within-patch population dynamics, the monotonicy of its spectral bound with respect to dispersal speed/coupling strength/travel frequency $\mu$ is established via two methods. The first method is an analytic derivation utilizing a graph-theoretic approach based on Kirchhoff's Matrix-Tree Theorem; the second method employs Collatz-Wielandt formula from matrix theory and complex analysis arguments. It turns out that our established result is a slightly strengthen version of Karlin-Altenberg's Theorem, which has previously been discovered independently while investigating reduction principle in evolution biology and evolution dispersal in patchy landscapes. Nevertheless, our result provides a new and effective approach in stability analysis of complex biological systems in a heterogeneous environment. We illustrate this by applying our result to well-known ecological models of single species, predator-prey and competition, and an epidemiological model of susceptible-infected-susceptible (SIS) type. We successfully solve some open problems in the literature of population dynamics.   \\[2mm]
\end{abstract}

\section{Introduction}
Various patch models have been proposed to investigate the impact of the environmental heterogeneity and the connectivity of subregions on the population dynamics. For example, the impact of the dispersal rate of susceptible and infected individuals among patches on the transmission of diseases has been studied in \cite{Allenpatch, arino2009diseases, arino2006disease, dhirasakdanon2007sharp, gao2019, gao2019fast, gao2012, LiPeng,li2010global,wang2004epidemic}; the evolution of dispersal in patchy environment has been shown to favor strategies resulting in ideal free distributions in \cite{CCDP2007,CCL2012,cantrell2017evolution}; and the persistence and extinction of predator and prey species in patchy environment have been considered in \cite{Freedman1989, holland2008strong,li2010global}.  

The movement pattern of individuals among $n$ subregions (or patches) can be described by a connectivity matrix $A=(a_{ij})_{n\times n}$, where $a_{ij}\ge 0$ ($i\neq j$) measures the movement  of individuals from patch $j$ to $i$ and $a_{ii}=-\ds\sum_{j\neq i} a_{ji}$ describes the total movement out from patch $i$. To study the effect of the connectivity of subregions on the population dynamics in a heterogeneous environment, one may consider a basic linear differential equation model
\begin{equation}\label{eq11}
\frac{d u_i(t)}{dt} = \mu \,\sum_{j=1}^n (a_{ij}u_j(t) - a_{ji}u_i(t))+ q_i u_i(t), \quad \quad i=1, \ldots, n,
\end{equation}
where $u_i(t)$ is the population size in the $i$-th patch, $\mu$ is the movement rate of individuals between patches and $q_i$ is the growth rate of the population in the $i$-patch. The growth or extinction of the population depends on the spectral bound $s(\mu A + Q)$ of matrix $\mu A + Q$, where $Q=\text{diag}(q_i)$ is a diagonal matrix. The dependence of the spectral bound $s(\mu A + Q)$ on the dispersal rate $\mu$ is of  significant importance in determining the population dynamics of the basic patch model \eqref{eq11} and other complex biological models based on it (e.g., those in Section~\ref{section-applications}).

Studies on the monotone dependence of $s(\mu A + Q)$ on $\mu$ started by Karlin \cite{karlin1982classifications}, and he  proved that $s(((1-\mu)I+\mu P)R)=s(\mu (P-I)R+R)$ is strictly decreasing in $\mu$ for $\mu\in (0, 1)$ unless $R$ is a multiple of the identity matrix $I$, where $P$ is a stochastic matrix and  $R$ is a positive diagonal matrix. Karlin's Theorem has been interpreted as the mathematical explanation of the \emph{reduction principle} \cite{altenberg2012resolvent, altenberg2017unified, feldman1980evolution} in evolutionary biology: greater mixing reduces growth. While studying the evolution of dispersal in patchy landscapes, Kirkland \emph{et al.} \cite{kirkland2006evolution}  independently discovered Karlin's Theorem with $P$ being a substochastic matrix.  More recently Altenberg \cite{altenberg2012resolvent} generalized Karlin's Theorem to linear operators on Banach spaces, and in context of matrix version, Altenberg's result actually showed that $s(\mu A + Q)$ is decreasing in $\mu$ when $A$ is quasi-positive with $s(A)\le 0$. Karlin's original proof in \cite{karlin1982classifications} utilizes the Donsker-Varadhan formula \cite{donsker1975variational} for principal eigenvalues of quasi-positive matrices, while Altenberg's proof in \cite{altenberg2012resolvent} relies on convex spectral functions due to Cohen \cite{cohen1981convexity}, Friedland \cite{friedland1981convex} and Kato \cite{kato1982superconvexity}. Kirkland \emph{et al.}'s proof \cite{kirkland2006evolution}  employs techniques from matrix analysis. 

In this paper, we provide two new different approaches to prove Karlin-Altenberg's Theorem, which also implies Karlin's original theorem.  Our first proof combines analytic method and a graph theory method  based on Kirchhoff's Matrix-Tree Theorem and Tree-Cycle Identity. The Tree-Cycle Identity has previously been used to construct Lyapunov functions for coupled systems of differential equations on a network \cite{guo2006global, guo2008graph, li2010global}, and here we show again  the Tree-Cycle Identity is an effective way to study the impact of network structure on the population dynamics. Specifically, we are able to show that $$\ds\frac{d}{d\mu} s(\mu A + Q)<0,\;\;\text{and}\;\;\ds\frac{d^2}{d\mu^2} s(\mu A + Q)>0$$ if $Q$ is not a multiple of $I$, where the result on  the strict positivity of the second derivative of $s(\mu A + Q)$  seems to be new. Our second proof to show that $s(\mu A + Q)$ is decreasing in $\mu$ uses only the  ``min-max" Collatz-Wielandt formula for the principal eigenvalue. We use a result from complex analysis to prove that $s(\mu A + Q)$ is either constant or strictly decreasing: the zeros of analytic functions are isolated; we also compute the limit of $s(\mu A + Q)$ as $\mu$ approaches zero or infinity.

The monotone dependence of $s(\mu A + Q)$ on $\mu$ can be applied to study the impact of the connectivity of subregions on the dynamics of  population models. In particular, we will consider the role of movement rate in single species model, predator-prey model, competition model, and epidemic SIS model in a multi-patch setting. Notably, for the competition model, we consider a situation that the two competing species are identical except for the movement rate, and we prove that the species with a slower dispersal rate will out-compete the one with faster dispersal rate. This result is in agreement with the corresponding reaction-diffusion model \cite{dockery1998evolution, hastings1983can}. For the epidemic SIS model, we prove the monotone dependence of the basic reproduction number on the movement rate of the infected population, which was addressed as an open problem in \cite{Allenpatch}.

Our paper is organized as follows. In Section 2, we list the terminology and state Karlin's Theorem.  In Sections 3 and 4, we give two different proofs of Karlin's Theorem. In Section 5, we will consider some population models from ecology and epidemiology and  study the impact of the movement rates of species between patches on the  population dynamics.


\section{Preliminary}
Let $A$ be an $n\times n$ matrix and let $\sigma(A)$ be the set of eigenvalues of $A$. Let $r(A)$ be the {\it spectral radius} of $A$, i.e.,
$$
r(A)=\max\{|\lambda|: \lambda\in\sigma(A)\}.
$$
Let $s(A)$ be the {\it spectral bound} (also called {\it spectral abscissa}) of $A$, i.e.,
$$
s(A)=\max\{{\rm Re} \lambda: \lambda\in\sigma(A)\}.
$$
A vector $u\gg0$ means that every entry of $u$ is positive. We use $I$ to denote the $n\times n$ identity matrix.

A square matrix is called  \emph{stochastic} if all the entries are nonnegative and every column adds up to $1$. Let $A=(a_{ij})$ be a square matrix; $A$ is called  an  \emph{$M$-matrix} if $a_{ij}\le 0$ for all $i\neq j$ and $A=sI-B$ with $B$ having all off-diagonal elements negative and $s\ge r(B)$; $A$ is called \emph{quasi-positive} (also called {\it Metzler matrix}) if $a_{ij}\ge 0$  for all $i\neq j$.
The connection between $M$-matrices and quasi-positive matrices can be viewed through the following well-known results: $-A$ is a non-singular (singular) $M$-matrix if and only if $A$ is quasi-positive with $s(A)<0$ ($s(A)=0$);  if $A$ is a singular $M$-matrix, then $A+D$ is a non-singular $M$-matrix for any  diagonal matrix $D=\text{diag}(d_i)$ with $d_i>0$ for all $i$.



A square matrix $L$ is called a \emph{Laplacian} matrix if all the off-diagonal entries are nonpositive and the sum of each column is zero. If $L$ is a Laplacian matrix, it is easy to see that $(1,1,\cdots,1)$ is a left eigenvector of $L$  corresponding to the eigenvalue $0$. In our applications to spatial population dynamics, the Laplacian matrix encodes all movement between patches and no population loss is assumed during the dispersal. For our purpose, a square matrix $L=(\ell_{ij})$ is called \emph{sub-Laplacian} if $\ell_{ij}\le 0$ for all $i\neq j$ and $\ell_{jj}\ge \ds -\sum_{i\ne j} \ell_{ij}$ for all $j$. A sub-Laplacian matrix $L$ is called \emph{strongly (strictly) sub-Laplacian} if $\ell_{jj}> \ds -\sum_{i\ne j} \ell_{ij}$ for all (some) $j$. Sub-Laplacian matrices defined above allow us to include possible population loss during the dispersal in our studies; see, for example, Section~\ref{section-single}.

In \cite{karlin1982classifications}, Karlin proved the following theorem on the monotonicity of the spectral radii of a family of  matrices, which was interpreted as the mathematical explanation of the \emph{reduction principle} \cite{altenberg2012resolvent, feldman1980evolution} in evolution biology. Karlin's proof relies on the Donsker-Varadhan formula for the principal eigenvalue. Later, this result has been discovered independently by Kirkland et al. \cite{kirkland2006evolution} (see also \cite{schreiber2009invasion}), and their proof is based on techniques of matrix analysis.

\begin{theorem}[Karlin's Theorem]\label{theorem_karlin}
Let $P$ be an irreducible stochastic matrix. Consider the family of matrices
$P_\mu=(1-\mu)I + \mu P$ with $0<\mu<1.$
Then for any diagonal matrix $R=\text{diag}(r_i)$ with $r_i>0$ for all $i$,
$r(P_\mu R)$ is strictly decreasing in $\mu$ provided that $R$ is not a multiple of $I$.
\end{theorem}

Theorem \ref{theorem_karlin} has been applied to the following discrete time linear population model \cite{karlin1982classifications}:
\begin{equation}\label{dlinear}
x(t+1)=[(1-\mu)I+\mu P]Rx(t).
\end{equation}
Here, $x(t)$ is a vector-valued function denoting the frequency of each subdivision of some population (e.g., genotypes); $R$ is a diagonal matrix measuring the growth rate of each subdivision; stochastic matrix $P$ represents the pattern of dispersal; $\mu$ is the rate of dispersal (or mutation, mixing, etc.). The spectral radius $r(P_\mu R)$ measures the growth rate of the population. Biologically, Theorem  \ref{theorem_karlin} implies that the evolution of population favors a smaller rate of dispersal.

To view the connection between the spectral radius problem on \eqref{dlinear} and our spectral bound problem on \eqref{eq11}, we set $P_\mu R=\mu (P-I)R+R=\mu A+R$, where $A=(P-I)R$ is a quasi-positive matrix. The corresponding continuous-time version of model \eqref{dlinear} can be written as
$$
x'(t)=[\mu A  + Q] x(t),
$$
where $Q=R-I$ is a diagonal matrix representing the growth rate of each subdivision, but the diagonal entry $q_i=r_i-1$ of $Q$ is not necessarily positive. Since $A$ is quasi-positive, it generates a positive semigroup $Exp(tA)$, which measures the dispersal (or mutation, mixing, etc.) between subdivisions. The impact of dispersal rate $\mu$ has been shown in the following  Karlin-Altenberg's Theorem:

\begin{theorem}[Altenberg \cite{altenberg2010karlin, altenberg2012resolvent}]\label{theorem_altenberg}
Let $A$ be an irreducible quasi-positive matrix and let $Q$ be a diagonal matrix. Consider the family of matrices
$M(\mu)= \mu A + Q$ with $\mu>0.$
Then
\begin{enumerate}
    \item $s(M(\mu))$ is either a constant or strictly decreasing in $\mu\in (0, \infty)$ if $s(A)\le 0$. Moreover,
    \begin{equation*}
\ds\frac{d}{d\mu}s(M(\mu)) \le s(A),
\end{equation*}
and the equality
holds if and only if  $Q$ is  a multiple of $I$;
    \item $s(M(\mu))$ is convex in $\mu$, i.e. for any $0<\alpha<1$, and $\mu_1,
    \mu_2\ge 0$ with $\mu_1\neq \mu_2$,
$$
s((1-\alpha)M(\mu_1)+\alpha M(\mu_2))\le (1-\alpha)s(M(\mu_1))+\alpha r(M(\mu_2)),
$$
and the equality holds if and only if  $Q$ is  a multiple of $I$.
\end{enumerate}
\end{theorem}

Theorem \ref{theorem_altenberg} can be used to prove Theorem \ref{theorem_karlin} (see the proof at the end of this section). In \cite{altenberg2012resolvent}, Altenberg uses the strict convexity of $r(P_\mu R)$ by Friedland \cite{friedland1981convex} to show the strict monotonicity of  $r(P_\mu R)$. Alternatively, one may compute the limits $\ds\lim_{\mu\rightarrow 0} r(P_\mu R)$ and $\ds\lim_{\mu\rightarrow \infty} r(P_\mu R)$: these two limits do not equal if and only if $R$ is not a multiple of $I$, and therefore Theorem \ref{theorem_altenberg} implies the strict monotonicity of $r(P_\mu R)$. We will use this idea in the proof of Theorem  \ref{theorem_quasi}. We remark that the original statement of Theorem \ref{theorem_altenberg} in \cite{altenberg2012resolvent} are for operators on Banach spaces. Altenberg's proof is based on the convexity of the spectral radius due to Cohen \cite{cohen1981convexity} and Kato \cite{kato1982superconvexity}.

Finally we prove that Theorem \ref{theorem_altenberg} implies Theorem \ref{theorem_karlin}. In the next two sections, we give two proofs of strengthened versions of Theorem \ref{theorem_altenberg}, which also lead to new proofs of Theorem \ref{theorem_karlin}.
\begin{proof}[Proof of Theorem \ref{theorem_karlin} from Theorem \ref{theorem_altenberg}]
Since $P$ is an irreducible stochastic matrix, $P_{\mu}=(1-\mu)I + \mu P = I - \mu(I-P)$ is a nonnegative irreducible matrix. Hence, $P_{\mu}R = R - \mu(I-P)R$ is nonnegative as $R$ is positive. It follows from the Perron-Frobeneius theory that $r(P_{\mu}R) = s(P_{\mu}R)$. On the other hand, all off-diagonal entries of $(I-P)R$ are non-positive, and the sum of entries of each column of $(I-P)R$ is zero. Hence $A=-(I-P)R$ is a Laplacian matrix, and thus quasi-positive with $s(A)=0$. Notice that $A$ is also irreducible since $P$ is irreducible. Therefore, by Theorem \ref{theorem_altenberg}, $r(P_{\mu} R)= s(P_{\mu}R)$ is strictly decreasing in $z$ provided that $R$ is not a multiple of $I$.
\end{proof}

\section{A graph theoretical proof of Karlin's theorem}

In this section, we will use a graph theoretical approach to prove Karlin-Altenberg's Theorem, and the results we obtain on the convexity are slightly stronger than Theorem \ref{theorem_altenberg}. The terminology and results from graph theory  can be found in the appendix.

Let $L$ be a Laplacian matrix and let $Q$ be a diagonal matrix. If $L$ is irreducible, then $Q-\mu L$ is irreducible and, by the Perron-Frobenius Theorem,  $s(Q-\mu L)$ is the principal eigenvalue of $Q-\mu L$, which is simple and associated with a positive eigenvector.  To study the spectral bound $s(Q-\mu L)$, we consider the weighted directed graph  with $n$ vertices associated with the Laplacian matrix $L$.


\begin{theorem} \label{thm2.1}
Let $L$ be an irreducible Laplacian matrix, and let $Q=\textrm{diag}(q_i)$ be a diagonal matrix. Denote $M(\mu)=Q-\mu L$.
Then for any $\mu>0$,
\begin{enumerate}
\item [$(i)$]
\begin{equation}\label{first}
\ds\frac{d}{d\mu}s(M(\mu)) \le 0,
\end{equation}
and the equality
holds if and only if  $q_1=q_2=\cdots=q_n$;
\item[$(ii)$]
\begin{equation}\label{second}
\ds\frac{d^2}{d\mu^2}s(M(\mu)) \ge 0,
\end{equation}
and the equality
holds if and only if  $q_1=q_2=\cdots=q_n$.
\end{enumerate}
\end{theorem}

\begin{proof}
By the Perron-Frobenius theorem, $s(M(\mu))$ is the principal eigenvlaue of $M=M(\mu)$. Denote $\lambda^*=s(M)=s(M^T)$, where $M^T$ is the transpose of $M$. Since $M^T$ is quasi-positive and irreducible, $\lambda^*$ is an eigenvalue of $M^T$ with corresponding eigenvector $w=(w_1,w_2,\ldots,w_n)^T$ with $w_j> 0$ for all $j$. Notice that $w_i$ and $\lambda^*$ depend smoothly on $\mu$.
Without loss of generality, we assume that
$\ds\sum_{i=1}^n w_i=1$ for any $\mu>0$, which implies that (here $'$ is the derivative with respect to $\mu$)
\begin{equation}\label{dzero}
\sum_{i=1}^n w_i'=0.
\end{equation}
Since $Q$ is diagonal, $M^T=Q-\mu L^T$. Hence, for each $i$,
\begin{equation}\label{sp1}
\lambda^* w_i = q_iw_i - \mu\sum_{j\not=i}a_{ji}w_i + \mu\sum_{j\not=i}a_{ji}w_j.
\end{equation}

We first prove $(i)$.  Differentiating \eqref{sp1} with respect to $\mu$ yields
\begin{equation}\label{sp2}
(\lambda^*)' w_i + \lambda^* w_i' = q_iw_i' - \sum_{j\not=i}a_{ji} w_i -\mu\sum_{j\not=i}a_{ji}w_i' + \sum_{j\not=i} a_{ji} w_j + \mu\sum_{j\not=i}a_{ji}w_j'.
\end{equation}
Multiplying  \eqref{sp2} by $w_i$ gives
\begin{equation}\label{sp3}
(\lambda^*)' w_i^2 + \lambda^* w_i'w_i = q_iw_i'w_i + \sum_{j\not=i}a_{ji} (w_j-w_i)w_i + \mu\sum_{j\not=i}a_{ji}(w_j'-w_i')w_i.
\end{equation}
By substituting  \eqref{sp1} into the second term of \eqref{sp3}, we obtain
\begin{equation}\label{sp4}
(\lambda^*)'w_i^2 = \sum_{j\not=i}a_{ji} w_jw_i\Big(1-\frac{w_i}{w_j}\Big) + \mu\sum_{j\not=i}a_{ji}w_jw_i\Big(\frac{w_j'}{w_j}-\frac{w_i'}{w_i}\Big).
\end{equation}
Now set $\bar{A}=(\bar{a}_{ij})_{n\times n}$ with $\bar{a}_{ij}=a_{ij}w_iw_j$ for $1\le i, j \le n$. Let $(\bar{\alpha}_1, \bar{\alpha}_2, \ldots, \bar{\alpha}_n)^T$ denote the positive, normalized principal right eigenvector of the Laplacian matrix corresponding to $\bar{A}$. Multiplying \eqref{sp4} by $\bar{\alpha}_i$ and summing these over all $i$ yield
\begin{equation}\label{sp5}
(\lambda^*)' \sum_{i=1}^n \bar{\alpha}_i w_i^2 = \sum_{i=1}^n \sum_{j\ne i} \bar{\alpha}_i \bar{a}_{ji}\Big[1-\frac{w_i}{w_j} + \mu\Big(\frac{w_j'}{w_j}-\frac{w_i'}{w_i}\Big)\Big].
\end{equation}
It follows from the Tree-Cycle identity (see Appendix) that
\begin{equation}\label{sp5b}
\begin{split}
&\sum_{i=1}^n \sum_{j\ne i} \bar{\alpha}_i \bar{a}_{ji}\Big[1-\frac{w_i}{w_j} + \mu\Big(\frac{w_j'}{w_j}-\frac{w_i'}{w_i}\Big)\Big]\\
=& \sum_{\mathcal{Q} \in \mathbb{Q}} w(\mathcal{Q}) \sum_{(s,r)\in E(\mathcal{C}_{\mathcal{Q}})} \Big[1-\frac{w_s}{w_r} + \mu\Big(\frac{w_r'}{w_r}-\frac{w_s'}{w_s}\Big)\Big],
\end{split}
\end{equation}
where $\mathbb{Q}$ is the set of all spanning unicycle graphs of $(\mathcal{G}, \bar{A})$, $w(\mathcal{Q})>0$ is the weight of $\mathcal{Q}$, and $\mathcal{C}_\mathcal{Q}$ denotes the directed cycle of $\mathcal{Q}$ with directed edge set $E(\mathcal{C}_{\mathcal{Q}})$. Along any directed cycle $\mathcal{C}_\mathcal{Q}$ of length $l$,
\begin{equation}\label{ineq}
\sum_{(s,r)\in E(\mathcal{C}_{\mathcal{Q}})} \Big(1-\frac{w_s}{w_r}\Big) = l - \Big(\sum_{(s,r)\in E(\mathcal{C}_{\mathcal{Q}})} \frac{w_s}{w_r}\Big) \le l - l\cdot \Big(\prod_{(s,r)\in E(\mathcal{C}_{\mathcal{Q}})} \frac{w_s}{w_r}\Big)^{1/l}=l-l\cdot 1 = 0.
\end{equation}
Here we use the AM-GM inequality $(x_1+x_2+\cdots+x_l)/l\ge \sqrt[l]{x_1x_2\cdots x_l}$
and
\begin{equation}\label{eq}
\sum_{(s,r)\in E(\mathcal{C}_{\mathcal{Q}})} \Big(\frac{w_r'}{w_r}-\frac{w_s'}{w_s}\Big) = 0.
\end{equation}
Combining \eqref{sp5}-\eqref{eq} yields $(\lambda^*)' \le 0$.
Notice that $(\lambda^*)'=0$ if and only if the equality holds in \eqref{ineq} for any directed cycle, that is, $w_r=w_s$ for any pair of $(s,r)$ locating in a directed cycle of $(\mathcal{G}, \bar{A})$. Since $\bar{A}$ is irreducible, $(\mathcal{G}, \bar{A})$ is strongly connected. As a consequence, $w_i=w_j$ for any $i,j$. Substituting these into \eqref{sp1} yields $\lambda^*=q_i$ for all $i$, which completes the proof of $(i)$.

Next we prove $(ii)$. In the following  $''$ is the second derivative with respect to $\mu$.
Differentiating \eqref{sp2} with respect to $\mu$ yields
\begin{equation}\label{sp22}
\begin{split}
&(\lambda^*)'' w_i + 2(\lambda^*)' w_i' +\lambda^* w_i''\\
=& q_iw_i'' - 2\sum_{j\not=i}a_{ji} w_i' -\mu\sum_{j\not=i}a_{ji}w_i'' + 2\sum_{j\not=i} a_{ji} w_j' + \mu\sum_{j\not=i}a_{ji}w_j''.
\end{split}
\end{equation}
Multiplying  \eqref{sp22} by $w_i$ gives
\begin{equation}\label{sp33}
\begin{split}
&(\lambda^*)'' w^2_i + 2(\lambda^*)' w_i'w_i +\lambda^* w_i''w_i\\
=& q_iw_iw_i'' - 2\sum_{j\not=i}a_{ji} w_i'w_i -\mu\sum_{j\not=i}a_{ji}w_i''w_i +2 \sum_{j\not=i} a_{ji} w_j'w_i + \mu\sum_{j\not=i}a_{ji}w_j''w_i.
\end{split}
\end{equation}
 Substituting \eqref{sp1} and \eqref{sp4} into \eqref{sp33},
we have
\begin{equation}\label{sp44}
\begin{split}
(\la^*)''w_i^2=&\mu\sum_{j\ne i}
a_{ji}w_jw_i\Big(\frac{w_j''}{w_j}-\frac{w_i''}{w_i}\Big)+2 \sum_{j\ne i} a_{ji}w_jw_i \Big(\frac{w_j'}{w_j}-\frac{w_i'}{w_i}\Big)\\
&-2\mu\sum_{j\ne i}a_{ji}w_jw_i\left[\frac{w_j'}{w_j}\frac{w_i'}{w_i}-\Big(\frac{w_i'}{w_i}\Big)^2\right].
\end{split}
\end{equation}
Recall $\bar{a}_{ij}=a_{ij}w_iw_j$. Multiplying \eqref{sp4} by $\bar{\alpha}_i$ and summing these over all $i$ yields
\begin{equation}\label{sp55}
\begin{split}
&(\la^*)''\sum_{i=1}^n\bar{\alpha}_iw_i^2\\
=&\sum_{i=1}^n\sum_{j\ne i}
\bar{\alpha}_i\bar{a}_{ji}\left[\mu\Big(\frac{w_j''}{w_j}-\frac{w_i''}{w_i}\Big)+2\Big(\frac{w_j'}{w_j}-\frac{w_i'}{w_i}\Big)-2\mu\left(\frac{w_j'}{w_j}\frac{w_i'}{w_i}-\Big(\frac{w_i'}{w_i}\Big)^2\right)\right]\end{split}
\end{equation}
It follows from the Tree-Cycle identity, \eqref{eq} and \eqref{eq} type equality for $w_j''/w_j$ that
\begin{equation}\label{sp5bb}
\begin{split}
&(\la^*)''\sum_{i=1}^n \bar\alpha_iw_i^2\\
=& \sum_{\mathcal{Q} \in \mathbb{Q}} w(\mathcal{Q}) \sum_{(s,r)\in E(\mathcal{C}_{\mathcal{Q}})} \left[\mu\Big(\frac{w_r''}{w_r}-\frac{w_s''}{w_s}\Big)+2\Big(\frac{w_r'}{w_r}-\frac{w_s'}{w_s}\Big)-2\mu\left(\frac{w_r'}{w_r}\frac{w_s'}{w_s}-\Big(\frac{w_s'}{w_s}\Big)^2\right)\right]\\
=& {\mu} \sum_{\mathcal{Q} \in \mathbb{Q}} w(\mathcal{Q}) \sum_{(s,r)\in E(\mathcal{C}_{\mathcal{Q}})}\left(\frac{w_r'}{w_r}-\frac{w_s'}{w_s}\right)^2 \ge 0.
\end{split}
\end{equation}
Notice that $(\lambda^*)''=0$ if and only if $\ds\frac{w'_r}{w_r}=\ds\frac{w_s'}{w_s}$ for any pair of $(s,r)$ locating in a directed cycle of $(\mathcal{G}, \bar{A})$. Since $\bar{A}$ is irreducible, the graph $(\mathcal{G}, \bar{A})$ is strongly connected. As a consequence, $\ds\frac{w_i'}{w_i}=\ds\frac{w_j'}{w_j}$ for any $i,j$. Therefore, $w_i'=kw_i$ for all $i$ for some $k\in\mathbb{R}$.
This, combined with \eqref{dzero} and $w_i>0$, implies that $w_i'=0$ for any $i=1,\dots,n$. Substituting $w_i'=0$ into
\eqref{sp2}, we have
\begin{equation*}
(\lambda^*)' w_i =- \sum_{j\not=i}a_{ji} w_i  + \sum_{j\not=i} a_{ji} w_j,
\end{equation*}
which implies that $(\lambda^*)'$ is the principal eigenvalue of $-L^T$ and therefore $(\lambda^*)'=0$.
From $(i)$ we see that $\la^*=q_i$ for all $i$,
and $(ii)$ holds.
\end{proof}

In Theorem \ref{thm2.1}, the column sum of the Laplacian matrix $L$ is zero which represents that the dispersal has no loss of population. A slightly stronger results hold for $\tilde{L}$ in which there is a loss of population when dispersing. Since Corollary \ref{cor2.2} follows directly from Theorem \ref{theorem_quasi0}, we do not prove it here.

\begin{corollary}\label{cor2.2}

Let $\tilde{L}=(\tilde{\ell}_{ij})_{n\times n}$ be an irreducible strictly sub-Laplacian matrix, let $Q=\textrm{diag}(q_i)$ be a diagonal matrix, and $\tilde{M}(\mu)=Q-\mu\tilde{L}$. Then for any $\mu>0$,
\begin{enumerate}
\item [$(i)$]
\begin{equation*}
\ds\frac{d}{d\mu}s(\tilde M(\mu)) < 0.
\end{equation*}
\item[$(ii)$]
\begin{equation*}
\ds\frac{d^2}{d\mu^2}s(\tilde M(\mu)) \ge 0,
\end{equation*}
and the equality holds if and only if $q_1=q_2=\dots=q_n$.
\end{enumerate}
\end{corollary}

Now we use the results in Theorem \ref{thm2.1} to prove the following version of Karlin-Altenberg Theorem (Theorem \ref{theorem_altenberg}).
\begin{theorem}\label{theorem_quasi0}
Let $A$ be a quasi-positive irreducible matrix  and let $Q=\text{diag}(q_i)$ be a diagonal matrix.  Then the following statements hold:
\begin{enumerate}
    \item[(i)]
    $$
    \frac{d}{d\mu} s(\mu A+ Q)\le s(A),
    $$
    and the equality holds if and only if  $q_1=q_2=\cdots=q_n$.
        \item[(ii)]
        $$
    \frac{d^2}{d\mu^2} s(\mu A+ Q)\ge 0,
    $$
    and the equality holds if and only if  $q_1=q_2=\cdots=q_n$.
\end{enumerate}
\end{theorem}
\begin{proof} First we assume that $s(A)= 0$. Since  $A$ is an irreducible quasi-positive matrix, by Perron-Frobenius theorem, $A$ has a left principal eigenvector $u=(u_1, u_2, \cdots, u_n)^T \gg 0$ corresponding with eigenvalue $s(A)$. Denote $U=\text{diag}(u_i)$ and $L=-UAU^{-1}$. Since $s(A)=0$,  $L$ is a Laplacian matrix.
Indeed since $UAU^{-1}=(u_ia_{ij}u_j^{-1})$, the matrix $UAU^{-1}$ is quasi-positive and the sum of the $j$-th column is $u_j^{-1}\ds\sum_{i=1}^n u_i a_{ij}=s(A) u_j^{-1}  u_j=s(A)$. If $s(A)=0$, then the sum of each column of $UAU^{-1}$ is zero and $-UAU^{-1}$ is a Laplacian matrix. Since $s(\mu A+Q)=s(U(\mu A+ Q)U^{-1})=s(Q-\mu L)$, the results follow from Theorem \ref{thm2.1}.


If $s(A)\ne 0$, we replace $A$ by $A-s(A)I$ to obtain
$$
  \frac{d}{d\mu} s(\mu A+ Q)\le s(A) \ \ \text{and} \ \  \frac{d^2}{d\mu^2} s(\mu A+ Q)\ge 0,
$$
and the equality holds if and only if $q_1=q_2=\cdots=q_n$.
\end{proof}

The non-increasing property of the spectral bound of irreducible matrices as established in Theorem \ref{theorem_quasi0} also holds for reducible matrices.
\begin{corollary}\label{cor2.3}
Let $A$ be a quasi-positive  matrix with $s(A)\le 0$, and let $Q=\textrm{diag}(q_i)$ be a diagonal matrix. Then  $s(\mu A+Q)$ is non-increasing and convex  for all $\mu>0$.
\end{corollary}
\begin{proof}
The eigenvalues and spectral bound of $A$ are invariant under permutation similarity transformation $
P^{-1}AP$ for a permutation matrix $P$. So without loss of generality, we can assume that $A$ is a block upper triangular matrix:
\begin{equation*}
    A=\begin{bmatrix}
  B_1 & * & \cdots & *  \\
  0& B_2 & \cdots & *\\
  \cdots& \cdots & \cdots &  \cdots\\
  0 &0   & \cdots & B_k\\
 \end{bmatrix}
\end{equation*}
where $B_i$ ($1\le i \le k$) are $p_i\times p_i$ irreducible quasi-positive matrices with size $p_i\ge 1$ and $\ds\sum_{i=1}^k p_i=n$. We also break $Q=\text{diag}(q_i)$ to match with the size of $A$: $Q=\textrm{diag}(Q_i)$, where $Q_i$ is a diagonal matrix of size $p_i\times p_i$. Apparently $s(\mu A+Q)=\max\{s(\mu B_i+Q_i):1\le i\le k\}$.

Since for each $i$, $s(\mu B_i+Q_i)$ is non-increasing and convex in $\mu$ from Theorem \ref{theorem_quasi0}, we conclude that $s(\mu A+Q)$ is also non-increasing and convex in $\mu$ as the maximum of a finite number of non-increasing and convex functions. Indeed we can have the strict inequality if either (i) $s(\mu A+Q)=s(\mu B_i+Q_i)$ for a fixed $1\le i\le k$ and $Q_i\ne c I_i$ for any $c\in \R$ and $I_i$ is the $p_i\times p_i$ identity matrix, or (ii) for each $1\le i\le k$, $Q_i\ne c I_i$ for any $c\in \R$. In either case, the strict inequality follows from Theorem \ref{theorem_quasi0}.
\end{proof}


One may suspect that the graph theoretical method can be used to show that the third derivative of $s(\mu A+ Q)$ is negative or positive. However, from following example, we can see that the third derivative may not be of one sign. Let
$$
L=
\begin{pmatrix}
\frac{1}{2} & -1 \\
-\frac{1}{2} & 1
\end{pmatrix}
\ \ \
\text{and}\ \ \
Q=
\begin{pmatrix}
1 & 0 \\
0 & 2
\end{pmatrix}.
$$
Then we can compute
$$
s(\mu):=s(Q-\mu L)= \frac{6-3\mu + \sqrt{9\mu^2-4\mu+4}}{4}.
$$
From elementary calculation,  we can see that $s(\mu)>0$, $s'(\mu)<0$ and $s''(\mu)>0$ for all $\mu>0$, which is in agreement with Theorem \ref{thm2.1}. But the third derivative $s^{(3)}(\mu)$ changes sign.

\section{A constructive proof of Karlin's theorem}\label{section-result}

In this section, we will use a different method to prove Karlin-Altenberg's Theorem (Theorem \ref{theorem_altenberg}).  Our proof is based on the ``min-max" Collatz-Wielandt formula:
$$
s(A)=\min_{u\gg0} \max_{1\le i\le n} \frac{[Au]_i}{[u]_i},$$
where $A=(a_{ij})_{n\times n}$ is a quasi-positive irreducible matrix. Our method to prove that  $s(\mu A+Q)$ is decreasing  in $\mu$ is elementary, and then we use a theory from complex analysis to prove that $s(\mu A+Q)$ is strictly decreasing: the zeros of analytic functions are isolated.

The following elementary algebra lemma is essential for the proof of  monotonicity of $s(\mu A+Q)$, which may be of independent interests.
\begin{lemma}\label{lemma_ess}
Let $\mu, \mu', u_i>0$, $i=1, 2, ..., n$. Suppose $u_i\neq u_j$, for all $i\neq j$, where $i, j=1, 2, ..., n$.
Then there exist $k_i>0$, $i=1, 2, ...n$, such that
\begin{equation}\label{ine}
\frac{u_j(\mu+\mu')}{\mu' u_i+\mu u_j}<\frac{k_i}{k_j}< \frac{\mu' u_j+\mu u_i}{u_i(\mu+\mu')}, \ \ \forall i\neq j, \ i, j=1, 2, ..., n.
\end{equation}
\end{lemma}
\begin{proof}
Without loss of generality, we may assume $u_1<u_2\cdots<u_n$. Firstly, we show that \eqref{ine} makes sense, i.e.
\begin{equation}\label{ine1}
\frac{u_j(\mu+\mu')}{\mu' u_i+\mu u_j}<\frac{\mu' u_j+\mu u_i}{u_i(\mu+\mu')}.
\end{equation}
Eq. \eqref{ine1} equivalent to
$$
u_iu_j(\mu+\mu')^2< (\mu' u_i + \mu u_j)(\mu' u_j+\mu  u_i),
$$
which can be simplified as
$$
2u_iu_j< u^2_i+u^2_j.
$$
Since $u_i\neq u_j$, \eqref{ine1} is true.

Now we construct $k_1, k_2, ..., k_n$. Let $k_1=1$. We choose $k_{i}>0$, $i\ge 2$, recursively, such that
\begin{equation}\label{ine2}
\frac{u_i(\mu+\mu')}{\mu' u_{i+1}+\mu u_i}<\frac{k_{i+1}}{k_i}< \frac{\mu' u_i+\mu u_{i+1}}{u_{i+1}(\mu+\mu')}, \ \ i=1, 2, ..., n-1.
\end{equation}
We only need to prove that $k_i$,  $i=1, 2, ..., n$, satisfy \eqref{ine}.

We claim that
\begin{equation}\label{i2}
\frac{u_i(\mu+\mu')}{\mu' u_{i+2}+\mu u_i}<\frac{k_{i+2}}{k_i}< \frac{\mu' u_i+\mu u_{i+2}}{u_{i+2}(\mu+\mu')}, \ \ i=1, 2, ..., n-2.
\end{equation}
To see this, by \eqref{ine2},
\begin{equation}\label{ine3}
\frac{u_{i+1}(\mu+\mu')}{\mu' u_{i+2}+\mu u_{i+1}}<\frac{k_{i+2}}{k_{i+1}}< \frac{\mu' u_{i+1}+\mu u_{i+2}}{u_{i+2}(\mu+\mu')}.
\end{equation}
Multiplying \eqref{ine2} and \eqref{ine3}, we obtain
\begin{equation}\label{ine4}
\frac{u_i(\mu+\mu')}{\mu' u_{i+1}+\mu u_i} \frac{u_{i+1}(\mu+\mu')}{\mu' u_{i+2}+\mu u_{i+1}}<\frac{k_{i+2}}{k_{i}}<  \frac{\mu' u_i+\mu u_{i+1}}{u_{i+1}(\mu+\mu')} \frac{\mu' u_{i+1}+\mu u_{i+2}}{u_{i+2}(\mu+\mu')}.
\end{equation}
To show \eqref{i2}, it suffices to prove that
\begin{equation}\label{med}
\frac{u_i(\mu+\mu')}{\mu' u_{i+2}+\mu u_i}<\frac{u_i(\mu+\mu')}{\mu' u_{i+1}+\mu u_i} \frac{u_{i+1}(\mu+\mu')}{\mu' u_{i+2}+\mu u_{i+1}}
\end{equation}
and
$$
\frac{d' u_i+\mu u_{i+1}}{u_{i+1}(\mu+\mu')} \frac{\mu' u_{i+1}+\mu u_{i+2}}{u_{i+2}(\mu+\mu')}< \frac{\mu' u_i+\mu u_{i+1}}{u_{i+2}(\mu+\mu')}.
$$
These two inequalities can be checked directly. We only show \eqref{med}, as the other is similar. Eq. \eqref{med} is equivalent to
$$
(\mu' u_{i+1}+\mu u_i)(\mu' u_{i+2}+\mu u_{i+1})<    (\mu' u_{i+2}+\mu u_i)u_{i+1}(\mu+\mu'),
$$
which can be simplified as
$$
\mu' \mu (u_{i+1}^2+u_iu_{i+2})<\mu' \mu(u_{i+1}u_{i+2}+u_{i}u_{i+1}).
$$
This is equivalent to
$$
\mu' \mu (u_{i+1}-u_{i+2})(u_{i+1}-u_{i})<0,
$$
which holds as $u_{i}<u_{i+1}<u_{i+2}$. This proves \eqref{i2}.

By \eqref{ine2} and \eqref{i2}, we can show that
\begin{equation}\label{i3}
\frac{u_i(\mu+\mu')}{\mu' u_{i+3}+\mu u_i}<\frac{k_{i+3}}{k_i}< \frac{\mu' u_i+\mu u_{i+3}}{u_{i+3}(\mu+\mu')}, \ \ i=1, 2, ..., n-3.
\end{equation}
The proof of \eqref{i3} is similar to \eqref{i2}. Indeed, by \eqref{ine2}, we have
\begin{equation*}
\frac{u_{i+2}(\mu+\mu')}{\mu' u_{i+3}+\mu u_{i+2}}<\frac{k_{i+3}}{k_{i+2}}< \frac{\mu' u_{i+2}+\mu u_{i+3}}{u_{i+3}(\mu+\mu')}.
\end{equation*}
Multiplying this inequality with \eqref{i2}, we can show \eqref{i3}.  Then by induction, we can show that
\begin{equation}\label{ij}
\frac{u_i(\mu+\mu')}{\mu' u_{i+j}+\mu u_i}<\frac{k_{i+j}}{k_i}< \frac{\mu' u_i+\mu u_{i+j}}{u_{i+j}(\mu+\mu')}, \ \ i=1, 2, ..., n-j, \ \ j=1, 2, ..., n.
\end{equation}
This proves \eqref{ine}.
\end{proof}

In the following, we prove the Karlin-Altenberg's Theorem via several steps. First we show that  $s(\mu A+Q)$ is non-increasing.

\begin{lemma}\label{lemma_quasi}
Let $A=(a_{ij})_{n\times n}$ be a quasi-positive irreducible matrix such that $-A^T$ is sub-Laplacian, and let $Q=\text{diag}(q_i)$ be a diagonal matrix.  Then $s(\mu A+Q)$ is non-increasing in $\mu\in (0, \infty)$. If, in addition, $-A^T$ is strongly sub-Laplacian, then $s(\mu A+Q)$ is strictly decreasing in $\mu\in (0, \infty)$.
\end{lemma}
\begin{proof}
By the Perron-Frobenius theorem, $A_\mu:=\mu A+Q$ has a positive eigenvector $u=(u_1, u_2, ..., u_n)^T$   corresponding with eigenvalue $\lambda_1=s(A_\mu)$. Then, we have
\begin{equation}\label{lambda1}
q_iu_i+\mu a_{i1}u_1+\mu a_{i2}u_2\cdots+\mu a_{in}u_n=\lambda_1 u_i, \ \ i=1, 2, ..., n.
\end{equation}
For any $\mu'>0$, to show $s(A_{\mu+\mu'})\le s(A_{\mu})$, by the Collatz-Wielandt formula, it suffices to find a strictly positive vector $v=(v_1, v_2, \cdots, v_n)^T$ such that
\begin{equation}\label{vv}
 \max_{1\le i \le n} \frac{[A_{\mu+\mu'} v]_i}{[v]_i}\le \lambda_1.
\end{equation}
Suppose $v_i=k_i u_i$  for some $k_i>0$, $i=1, 2, ..., n$. We need to find $k_i$ satisfying \eqref{vv}, i.e.
\begin{equation}\label{lambda2}
\frac{q_ik_iu_i+(\mu+\mu')a_{i1}k_1u_1+(\mu+\mu')a_{i2}k_2u_2\cdots+(\mu+\mu')a_{in}k_nu_n}{k_i u_i}\le \lambda_1, \ i=1, 2, ..., n.
\end{equation}
Solving $\lambda_1$ from \eqref{lambda1} and plugging it into \eqref{lambda2}, \eqref{lambda2} is equivalent to
\begin{equation}\label{lambda3}
a_{ii} \mu' k_i u_i\le  \sum_{j\neq i}a_{ij}u_j(\mu k_i-(\mu+\mu')k_j),\ \forall i=1, 2,  ..., n,
\end{equation}
Since $|a_{ii}|\ge \sum_{j\neq i} a_{ij}$ as $-A^T$ is sub-Laplacian,  \eqref{lambda3} holds if
$$
\mu' k_i u_i\ge  u_j((\mu+\mu')k_j-\mu k_i), \ \forall j\neq i,
$$
 which is equivalent to
\begin{equation}\label{kkk}
\frac{k_i}{k_j}\ge  \frac{u_j(\mu+\mu')}{\mu' u_i+\mu u_j}, \ \forall j\neq i.
\end{equation}
By Lemma \ref{lemma_ess}, we can find $k_i$ satisfying \eqref{kkk} (if $u_i=u_j$, we may set $k_i=k_j$).  This proves  $r(A_{\mu+\mu'})\le r(A_\mu)$.

If $-A^T$ is strongly sub-Laplacian, then the inequality \eqref{lambda2} is strict and $r(A_\mu)$ is strictly decreasing.
\end{proof}

Next, we show that $s(\mu A+Q)$ is analytic in $\mu$. Since the zeros of analytic functions are isolated and $s(\mu A+Q)$ is decreasing, $s(\mu A+Q)$ is either strictly decreasing or constant in $\mu$.

\begin{lemma}\label{lemma_quasi2}
Let $A=(a_{ij})_{n\times n}$ be a quasi-positive irreducible matrix such that $-A^T$ is sub-Laplacian, and $Q=\text{diag}(q_i)$ be a diagonal matrix.  Then $s(\mu A+Q)$ is either strictly decreasing or constant in $\mu\in (0, \infty)$.
\end{lemma}
\begin{proof}
By the Perron-Frobenius theorem, $s^*(\mu):=s(\mu A+Q)$ is a simple root  of some polynomial equation $F(\mu, s)=0$ for each $\mu\in (0, \infty)$. Therefore,
$$
\frac{\partial F}{\partial s} (\mu, s^*)\neq 0 \text{ for all } \ \mu\in (0, \infty).
$$
By the implicit function theorem, $s^*(\mu)$ is analytic in $\mu$ (we may extend the domain of $\mu$ and $s$ to the complex plane). Since the zeros of analytic functions are isolated and $s(\mu A+Q)$ is decreasing by Lemma \ref{lemma_quasi}, $s(\mu A+Q)$ is either strictly decreasing or constant in $\mu$.
\end{proof}

Now we are ready to establish monotonicity of the spectral bound $s(\mu A+Q)$.

\begin{theorem}\label{theorem_quasi}
Let $A=(a_{ij})_{n\times n}$ be a quasi-positive irreducible matrix and let $Q=\text{diag}(q_i)$ be a diagonal matrix.   Then the following results hold:
\begin{enumerate}
\item If $s(A)<0$, then $s(\mu A+Q)$ is strictly decreasing in $\mu\in (0, \infty)$. Moreover   $$\lim_{\mu\rightarrow 0}s(\mu A+Q)=\max_{1\le i\le n}\{q_i\} \text{  and }\; \lim_{\mu\rightarrow\infty}s(\mu A+Q)=-\infty.$$
\item If $s(A)=0$, then $s(\mu A+Q)$ is strictly decreasing provided that $Q$ is not a multiple of $I$. Moreover,
$$
\lim_{\mu\rightarrow 0}s(\mu A+Q)=\max_{1\le i\le n}\{q_i\} \;\;
\text{and} \;\; \lim_{\mu\rightarrow\infty}s(\mu A+Q)=\sum_{i=1}^n{v_iq_i},
$$
where $v_i\in (0, 1)$ for each $1\le i\le n$ is determined by $A$ and satisfies $\ds\sum_{i=1}^n{v_i}=1$ (if $A$ has each row sum equaling zero, then $v$ is a left positive eigenvector of $A$).
\end{enumerate}
\end{theorem}
\begin{proof}
It is easy to see that $\ds\lim_{\mu\rightarrow 0}s(\mu A+Q)=\max_{1\le i\le n}\{q_i\}$. Let $u=(u_1, u_2, ..., u_n)^T$  be the positive eigenvector of $A$ corresponding to $s(A)$ satisfying $\ds\sum_{i=1}^n{u_i}=1$ and let $U=\text{diag}(u_i)$.

If  $s(A)<0$, then $\tilde A:=U^{-1}AU$ is quasi-positive, and $-\tilde A^T$ is strongly sub-Laplacian. By Lemma \ref{lemma_quasi}, $s(\mu A+Q)=s(\mu \tilde A+Q)$ is strictly decreasing. Since
\begin{equation}\label{eigb}
\min_{1\le i\le n}\left\{ \mu\sum_{j=1}^n a_{ij}+q_i \right\} \le s(\mu A+Q)\le  \max_{1\le i\le n}\left\{ \mu\sum_{j=1}^n a_{ij}+q_i\right \},
\end{equation}
and $\ds\sum_{j=1}^n a_{ij}<0$ for each $i$, we have $\ds\lim_{\mu\rightarrow\infty}s(\mu A+Q)=-\infty$.

If $s(A)=0$, then $\tilde A$ is quasi-positive and $-\tilde A^T$ is Laplacian.  By Lemma \ref{lemma_quasi2}, $s(\mu A+Q)=s(\mu \tilde A+Q)$ is strictly decreasing or constant in $\mu$.  By \eqref{eigb}, $s(\mu
\tilde A+Q)$ is bounded below by $\ds\min_{1\le i\le n}\{q_i\}$ and above by $\ds\max_{1\le i\le n}\{q_i\}$. Therefore, $\ds\lim_{\mu\rightarrow\infty}s(\mu \tilde A+Q)$ exists. Let $v_\mu=(v_{\mu1}, v_{\mu2}, ..., v_{\mu n})$ be the left positive eigenvector satisfying $\ds\sum_{i=1}^n v_{\mu i}=1$ for $\mu \tilde A+Q$. Up to a subsequence, we have $v_\mu\rightarrow v=(v_1, v_2, ..., v_n)$ as $\mu\rightarrow\infty$ for some nonnegative vector $v$ satisfying $\ds\sum_{i=1}^n v_i=1$. Dividing both sides of  $v_\mu(\mu \tilde A+Q)=s(\mu \tilde A+Q)v_\mu$ by $\mu$ and taking $\mu\rightarrow\infty$, we obtain $v\tilde A=0$. Therefore, $v$ is the left positive eigenvector of $\tilde A$ satisfying $\ds\sum_{i=1}^n v_i=1$. Summing up all the equations of $v_\mu(\mu \tilde A+Q)=s(\mu  \tilde A+Q)v_\mu$ and using the fact that $\tilde A$ has each row sum equaling zero, we obtain
$$
\sum_{i=1}^n {v_{\mu i} q_i}=s(\mu \tilde A+Q)\sum_{i=1}^n {v_i}=s(\mu \tilde A+Q).
$$
Taking $\mu\rightarrow\infty$, we have
$$
\lim_{\mu\rightarrow\infty}s(\mu \tilde A+Q)=\sum_{i=1}^n{v_iq_i}.
$$
Since $\ds\sum_{i=1}^n{v_iq_i}<\max_{1\le i\le n}\{q_i\}$ if and only if $(q_1, q_2, ..., q_n)$ is not a multiple of $(1, 1, ..., 1)$, $s(\mu \tilde A+Q)$ is strictly decreasing if $Q$ is not a multiple of $I$.
\end{proof}

\section{Applications}\label{section-applications}
In this section we apply Theorems \ref{theorem_quasi0} and \ref{theorem_quasi} to study several population models in heterogeneous environment from ecology and epidemiology.



\subsection{Single species model}\label{section-single}
We consider a general single species model in a heterogeneous environment of $n$ patches ($n\ge 2$)
\begin{equation}\label{single}
u_i' = u_i f_i(u_i) + \mu \,\sum_{j=1}^n (a_{ij}u_j - a_{ji}u_i)-\mu \epsilon_i u_i, \quad \quad i=1, \ldots, n,
\end{equation}
where $u_i$ denotes the population size (or density) in patch $i$; function $f_i$ denotes the intrinsic growth rate in patch $i$; the connectivity matrix $A=(a_{ij})_{n\times n}$ describes the dispersal pattern between patches, where $a_{ij}\ge 0$  for $i\not=j$ quantifies the dispersal from patch $j$ to patch $i$, and $a_{jj}=-\ds\sum_{i\neq j} a_{ij}$  is the total movement out from patch $j$;  $\mu \ge 0$ is the dispersal rate, and $\epsilon_i\ge 0$ is the death rate due to dispersal. When $\ep_i=0$ for all $i$,  there is no loss of individuals during the movement between patches; and when $\ep_i>0$ for some $1\le i\le n$, there are losses of individuals during the movement between patches.  
The intrinsic growth function $f_i$ ($1\le i\le n$) satisfies
\begin{itemize}
\item[(f)] For $1\le i\le n$, $f_i:\mathbb{R}_+\to\mathbb{R}$ is continuous and strictly decreasing. Moreover, there exists $c_i>0$ such that $f_i(u_i)<0$ for all $u_i>c_i$.
\end{itemize}

System \eqref{single} admits a trivial equilibrium $E_0=(0,0,\ldots,0)$, representing the state of species extinction, and its stability can be determined by the Jacobian matrix $J=\mu (A-{\rm diag}(\epsilon_i))+\textrm{diag}(f_i(0))$.



Assume that $A$ is irreducible.   It follows from Perron-Frobenius theorem that $A-{\rm diag}(\epsilon_i)$ has a principal eigenvector $(\alpha_1, \alpha_2, \ldots, \alpha_n)^T >0$ corresponding to the principal eigenvalue $0$ such that $\alpha_i >0$ for all $i$. As shown in Theorem \ref{theorem_quasi}, $\alpha_i$ serves as the weight constant in determining the impact of patch $i$ on the dynamics of the interconnected system \eqref{single} for larger value of $\mu$. The following result describes this phenomenon in detail.

\begin{theorem}\label{thm-single}
Suppose that $A$ is an irreducible quasi-positive matrix, and assume {\rm (f)} is satisfied. Let $(\alpha_1, \alpha_2, \ldots, \alpha_n)^T >0$ denote the normalized eigenvector of $A-{\rm diag}(\epsilon_i)$ corresponding to the principal eigenvalue $0$ such that $\ds\sum_{i=1}^n{\alpha_i}=1$.  Let $M=\ds\max_{1\le i\le n}\{f_i(0)\}$ and $m=\ds\sum_{i=1}^n \alpha_i f_i(0)$. Then the following statements hold:
\begin{enumerate}
    \item[{\rm (i)}] If $M<0$, then the equilibrium $E_0$ is globally asymptotically stable in $\mathbb{R}_+^n$ for all $\mu>0$.
    \item[{\rm (ii)}]  If $\epsilon_i=0$ for all $1\le i\le n$ and $m>0$, then the equilibrium $E_0$ is  unstable for all  $\mu>0$. Furthermore, system \eqref{single} admits a unique positive equilibrium $E^*=(u_1^*, \ldots, u_n^*)$, which is globally asymptotically stable in $\mathbb{R}_+^n-\{E_0\}$.
    \item[{\rm (iii)}] If either $\epsilon_i=0$ for all $1\le i\le n$ and $m<0<M$, or  $\epsilon_i>0$ for at least some $i\in \{1,\cdots,n\}$ and $M>0$, then there exists a unique $\mu^*>0$ such that $E_0$ is globally asymptotically stable in $\mathbb{R}_+^n$  for $\mu\ge\mu^*$ while $E_0$ is unstable for $0<\mu<\mu^*$. Furthermore, if $0<\mu<\mu^*$, then there exists a unique positive equilibrium $E^*$, which is globally asymptotically stable in $\mathbb{R}_+^n-\{E_0\}$.
\end{enumerate}
\end{theorem}

\begin{proof}
The local stability of $E_0$ is determined by the sign of the spectral bound of the Jacobian matrix $J(\mu) =\mu (A-{\rm diag}(\epsilon_i))+\textrm{diag}(f_i(0))$, i.e. $E_0$ is locally asymptotically stable if $s(J(\mu))<0$ and it is unstable if $s(J(\mu))>0$.  By Theorem \ref{theorem_quasi}, if $\epsilon_i=0$ for all $1\le i\le n$, $s(J(\mu))$ is non-increasing for $\mu\in (0, \infty)$ with
$$
m=\lim_{\mu\rightarrow\infty} s(J(\mu)) \le s(J(\mu))\le \lim_{\mu\rightarrow 0} s(J(\mu))=M,
$$
and if $m< M$ then  $s(J(\mu))$ is strictly decreasing; if  $\epsilon_i>0$ for at least some $i\in \{1,\cdots,n\}$, $s(J(\mu))$ is strictly decreasing in $\mu$ with
$$
-\infty=\lim_{\mu\rightarrow\infty} s(J(\mu)) < s(J(\mu))< \lim_{\mu\rightarrow 0} s(J(\mu))=M.
$$
The claim on the local stability of $E_0$ follows from these observations.

Since $A$ is irreducible and quasi-positive, \eqref{single} generates a strongly monotone dynamical system \cite{smith2008monotone}. Since $f_i$ is strictly decreasing for $1\le i\le n$,  the semiflow generates by  \eqref{single} is strongly sub-homogeneous \cite{zhao2003dynamical} (i.e. $\lambda T(t){ u}_0\ll T(t)\lambda { u}_0$ for all $\lambda\in (0, 1)$ and initial data ${ u}_0\gg0$, where $T(t)$ is the semiflow induced by \eqref{single}). Moreover, by the assumption (f), the solutions of \eqref{single} are bounded and dissipative (i.e. uniformly ultimately bounded).  Therefore, by \cite[Theorem 2.3.4]{zhao2003dynamical}, if  $s(J(\mu))\le 0$, $E_0$ is globally stable; if $s(J(\mu))> 0$, $E_0$ is unstable and there exists a unique globally stable positive steady state $E^*$. This completes the proof.
\end{proof}

\begin{remark}
\begin{enumerate}
    \item If the maximum growth rates $f_i(0)$ at the $i$-th patch are not all identical, then $m<M$.
    \item The $i$-th patch is a sink if $f_i(0)\le 0$ and it is a source if $f_i(0)>0$. If all patches are source, then apparently $m>0$, but the population can still become extinct for large $\mu$ if there is population loss during dispersal. If some patches are sink, then it is possible that $m<0$ but it still depends on the network connection. when the sink patches carries a larger weight $\alpha_i$, it is more likely that $m<0$.
    \item The extinction/persistence dichotomy of dynamics in terms of stability of extinction state  and the global stability of positive equilibrium of  \eqref{single} are  well-known, see for example \cite{cosner1996variability,LiShuai,Lu1993}.  Theorem \ref{thm-single} shows how the  extinction or persistence of population depends on the dispersal coefficient $\mu$.
    \item A similar result for reducible $A$ can also be obtained by using Corollary \ref{cor2.3} and the approach in \cite{Du2014}.
\end{enumerate}
\end{remark}

\subsection{Predator-prey model}
We consider the following predator-prey model with a general functional response in heterogeneous environment of $n$ patches ($n\ge2$):
\begin{equation}\label{predator-prey}
\begin{cases}
u_i'=r_iu_i\left(1-\displaystyle\frac{u_i}{K_i}\right)-g_i(u_i)v_i+\mu_u\ds\sum_{j=1}^{n}(a_{ij}u_j-a_{ji}u_i),&i=1,2,\dots,n,\\
v_i'=v_i(c_ig_i(u_i)-d_i)+\mu_v\ds\sum_{j=1}^n (b _{ij}v_j-b_{ji}v_i),&i=1,2,\dots,n,\\
\end{cases}
\end{equation}
where $u_i$ and $v_i$ denotes the population density of the prey and predators in the $i$-th patch, respectively; $r_i, K_i> 0$ are the growth rate and carrying capacity of the prey in the $i$-th patch, respectively;
 $d_i$ is the mortality rate of the predator, and $c_i$ is the conversion rate of the predation; the connectivity matrices $A=(a_{ij})_{n\times n}$ and $B=(b_{ij})_{n\times n}$ describe the dispersal pattern between patches for prey and predators respectively, where $a_{ij}\ge 0$ and $b_{ij}\ge 0$, $i\not=j$,  denote rate of the prey and predators from patch $j$ to patch $i$, and $a_{jj}=-\ds\sum_{i\neq j} a_{ij}$ and $b_{jj}=-\ds\sum_{i\neq j} b_{ij}$ are the total movement out from patch $j$ of the prey and predators, respectively; $\mu_u,\mu_v \ge 0$ denote the rates of dispersal of the two species $u$ and $v$, respectively. Function $g_i$ denotes the functional response of  predator in the $i$-th patch and satisfies the following assumption.
 \begin{enumerate}
     \item[(g)] For $1\le i\le n$, $g_i:\mathbb{R}_+\to\mathbb{R_+}$ is continuous, strictly increasing and $g_i(0)=0$.
 \end{enumerate}

The following result highlights the impact of dispersal rates on population dynamics of \eqref{predator-prey}.
\begin{theorem}\label{thm-prey}
Suppose that $A$ and $B$ are irreducible matrices, and assume {\rm (g)} is satisfied. Let $(\alpha_1, \alpha_2, \ldots, \alpha_n)^T$ be the positive eigenvector of $B$ corresponding to  eigenvalue $0$ with $\ds\sum_{i=1}^n{\alpha_i}=1$.  Then for any $\mu_u> 0$, $\mu_v> 0$, system \eqref{predator-prey} admits a trivial equilibrium $E_0=(0, 0, \dots, 0)$ and a unique semitrivial equilibrium $E_1=(u^*_1,\dots,u^*_n,0,\dots,0),$ where $u_i^*>0$ and satisfies
\begin{equation}\label{singleu}
r_iu^*_i\left(1-\displaystyle\frac{u^*_i}{K_i}\right)+\mu_u\sum_{j=1}^{n}(a_{ij}u^*_j-a_{ji}u^*_i)=0,\ \ j=1,2,\dots,n.
\end{equation}
Denote $M=\ds\max_{1\le i\le n}\{c_ig_i(u_i^*)-d_i\}$ and $m=\ds\sum_{i=1}^n \alpha_i \left(c_ig_i(u_i^*)-d_i\right)$.
Then the following statements hold:
\begin{enumerate}
    \item[{\rm (i)}] $E_0$ is unstable for any $\mu_v> 0$.
    \item[{\rm (ii)}] If $M<0$, then the equilibrium $E_1$ is  globally asymptotically stable for all $\mu_v> 0$.
    \item[{\rm (iii)}] If $m>0$, then the equilibrium $E_1$ is  unstable for all $\mu_v> 0$.
    \item[{\rm (iv)}] If $m<0<M$, then there exists a unique $\mu_v^*>0$ such that $E_1$ is globally asymptotically stable for $\mu_v>\mu_v^*$  while $E_1$ is unstable for $0<\mu_v<\mu_v^*$.
\end{enumerate}
\end{theorem}
\begin{proof}
The existence and uniqueness of $E_1$ follow from Theorem \ref{thm-single}. We prove the local stability/instability of $E_1$ in (ii)-(iv), as the proof of (i) is similar and simpler.
Linearizing \eqref{predator-prey} at $E_1$, the local stability of $E_1$ is determined by the following eigenvalue problem:
\begin{equation}\label{eigp}
\begin{cases}
\lambda \phi_i=r_i\phi_i\left(1-2\displaystyle\frac{ u^*_i}{K_i}\right)-g(u_i^*)\psi_i+\mu_u\ds\sum_{j=1}^{n}(a_{ij}\phi_j-a_{ji}\phi_i), &i=1,2,\dots,n,\\
\lambda \psi_i=\psi_i(c_ig(u_i^*)-d_i)+\mu_v\ds\sum_{j=1}^{n}(a_{ij}\psi_j-a_{ji}\psi_i),&i=1,2,\dots,n.
\end{cases}
\end{equation}
If ${\rm Re}\lambda<0$ for any eigenvalue $\lambda$ of \eqref{eigp}, then $E_1$ is locally asymptotically stable; if \eqref{eigp} has an eigenvalue $\lambda$ with  ${\rm Re}\lambda>0$, then $E_1$ is unstable.

We claim that the local  stability of $E_1$ is determined by the sign of $s(\mu_vA+\textrm{diag}(c_ig_i(u_i^*)-d_i))$. To see this, let $(\phi, \psi)$ with $\phi=(\phi_1, \phi_2, ..., \phi_n)^T$ and $\psi=(\psi_1, \psi_2, ..., \psi_n)^T$ be an eigenvector of \eqref{eigp} corresponding to eigenvalue $\lambda$. If $\psi=  0$, then $\lambda$ is an eigenvalue of
$$
\lambda \phi_i=r_i\phi_i\left(1-2\displaystyle\frac{ u^*_i}{K_i}\right)+\mu_u\ds\sum_{j=1}^{n}(a_{ij}\phi_j-a_{ji}\phi_i), \ \ i=1,\dots,n,
$$
i.e. an eigenvalue of $\mu_uA+\text{diag}(r_i(1-2u_i^*/K))$. By \eqref{singleu} and Perron-Frobenius Theorem, $s(\mu_uA+\text{diag}(r_i(1-u_i^*/K)))=0$. Therefore, $s(\mu_uA+\text{diag}(r_i(1-2u_i^*/K)))<s(\mu_uA+\text{diag}(r_i(1-u_i^*/K)))=0$. 
Hence, we have
\begin{equation}\label{lambda11}
{\rm Re}\lambda\le s(\mu_uA+\text{diag}(r_i(1-2u_i^*/K)))<0.
\end{equation}
If $\psi\neq  0$, $\lambda$ is an eigenvalue of
$$
\lambda \psi_i=\psi_i(c_ig(u_i^*)-d_i)+\mu_v\ds\sum_{j=1}^{n}(a_{ij}\psi_j-a_{ji}\psi_i), \ \ i=1,\dots,n,
$$
i.e. $\lambda$ is an eigenvalue of $\mu_vA+\textrm{diag}(c_ig_i(u_i^*)-d_i)$. Noticing \eqref{lambda11}, we see that the local stability of $E_1$ is determined by the sign of $s(\mu_vA+\textrm{diag}(c_ig_i(u_i^*)-d_i))$. Then the results (ii)-(iv) on the local stability of $E_1$ follow from the claim and Theorems \ref{theorem_quasi0} and \ref{theorem_quasi}.

It remains to prove the global stability of $E_1$ when $s(\mu_vA+\textrm{diag}(c_ig_i(u_i^*)-d_i))<0$. Suppose that $(u_1(0), u_2(0), ..., u_n(0))$ is nontrivial.  Let $\hat u_i(t)$, $1\le i\le n$, be the solution of
\begin{equation*}
\begin{cases}
\hat u_i'=r_i\hat u_i\left(1-\displaystyle\frac{\hat u_i}{K_i}\right)+\mu_u\ds\sum_{j=1}^{n}(a_{ij}\hat u_j-a_{ji}\hat u_i),\ \ &i=1,2,\dots,n,\\
\hat u_i(0)=u_i(0), \ \ &i=1,2,\dots,n.
\end{cases}
\end{equation*}
By the comparison principle, we have $u_i(t)\le
\hat u_i(t)$ for all $t\ge 0$ and $1\le i\le n$. By Theorem \ref{thm-single}, we have $\ds\lim_{t\rightarrow\infty} \hat u_i(t)=u_i^*$, and it follows that
$\ds\limsup_{t\rightarrow\infty} \hat u_i(t)=u_i^*$
for $1\le i\le  n$.
Choose $\epsilon_0>0$ such that $s(\mu_vA+\textrm{diag}(c_ig_i(u_i^*+\epsilon_0)-d_i))<0$. Then there exists $T>0$ such that $u_i(t)\le u_i^*+\epsilon_0$ for all $t\ge T$. By the second equation of \eqref{predator-prey} and the monotonicity of $g_i$, we have
\begin{equation*}
\begin{cases}
v_i'\le v_i(c_ig(u_i^*+\epsilon_0)-d_i)+\mu_v\ds\sum_{j=1}^n (b _{ij}v_j-b_{ji}v_i), \ \ &t\ge T, \; i=1,2,\dots,n,\\
v_i(T)\le C \tilde\alpha_i, \ \ &t\ge T, \; i=1,2,\dots,n,
\end{cases}
\end{equation*}
where $(\tilde\alpha_1, \tilde\alpha_2, ..., \tilde\alpha_n )$ is a positive principal eigenvector of $\mu_vA+\textrm{diag}(c_ig_i(u_i^*+\epsilon_0)-d_i)$ corresponding with eigenvalue $s_0:=s(\mu_vA+\textrm{diag}(c_ig_i(u_i^*+\epsilon_0)-d_i))$ and $C>0$ is large. By the comparison principle, we have $v_i(t)\le \hat v_i(t)$ for $t\ge T$, where $\hat v_i$ is the solution of the problem
\begin{equation}\label{cc1}
\begin{cases}
\hat v_i'= \hat v_i(c_ig(u_i^*+\epsilon_0)-d_i)+\mu_v\ds\sum_{j=1}^n (b _{ij}\hat v_j-b_{ji}\hat v_i), \ \ &t\ge T, \; i=1,2,\dots,n,\\
\hat v_i(T)= C \tilde\alpha_i, \ \ &t\ge T,\; i=1,2,\dots,n.
\end{cases}
\end{equation}
It is easy to check that the solution of \eqref{cc1} is $\hat v_i(t)= C\tilde \alpha_i e^{s_0(t-T)}$, $1\le i\le  n$. Since $s_0<0$, we have $\ds\lim_{t\rightarrow\infty}
\hat v_i(t)=0$, which implies $\ds\lim_{t\rightarrow\infty}
 v_i(t)=0$. Finally by the theory of asymptotically
autonomous semiflows (see, e.g., \cite{thieme1992convergence}) and Theorem \ref{thm-single}, we have
$\ds\lim_{t\rightarrow\infty}
u_i(t)=u_i^*$, $1\le i\le  n$.
\end{proof}

\begin{remark}
\begin{enumerate}
    \item When $\mu_v=0$, then $s(\mu_vA+\textrm{diag}(c_ig_i(u_i^*)-d_i))=M$ so part {\rm (ii)} in Theorem \ref{thm-prey} still holds.
    \item When $E_1$ is unstable, one can show the existence of a coexistence equilibrium $E_2$ through the theory of uniform persistence. When the functional response $g_i$ is of Lotka-Volterra type ($g_i(u_i)=u_i$), $E_2$ can be shown to be globally asymptotically stable when $\mu_v=0$ (see \cite[Theorem 6.1]{li2010global}). But when $g_i$ is of Monod type ($g_i(u_i)=u_i/(a_i+u_i)$), \eqref{predator-prey} is an $n$-patch Rosenzweig-MacArthur predator-prey system, $E_2$ may be unstable and the system could have a limit cycle even in the $1$-patch case.
    \item In Theorem \ref{thm-prey}, the growth rate $r_i$ for the prey is assumed to be positive in all patches. If $r_i$ are not all positive, then from Theorem \ref{thm-single}, a unique critical prey dispersal rate $\mu_u^*>0$ may exist so that $E_0$ is globally asymptotically stable for $\mu_u>\mu_u^*$  while $E_0$ is unstable for $0<\mu_u<\mu_u^*$. In that case, results {\rm (ii)-(iv)} in Theorem \ref{thm-prey} hold for $0<\mu_u<\mu_u^*$.
\end{enumerate}
\end{remark}

\subsection{Lotka-Volterra competition model}\label{section5.3}
We consider the following Lotka-Volterra competition model in a heterogeneous environment of $n$ patches ($n\ge2$):
\begin{equation}\label{competition}
\begin{cases}
u_i'=u_i(p_i-u_i-v_i)+\mu_u\ds\sum_{j=1}^{n}(a_{ij}u_j-a_{ji}u_i), &i=1,\dots,n,\\
v_i'=v_i(p_i-u_i-v_i)+\mu_v\ds\sum_{j=1}^{n}(a_{ij}v_j-a_{ji}v_i),&i=1,\dots,n,\\
u(0)=u_0\ge(\not\equiv)0,\;v(0)=v_0\ge(\not\equiv)0,
\end{cases}
\end{equation}
where
$u=(u_1,\dots,u_n)$ and $v=(v_1,\dots,v_n)$, and $u_i$ and $v_i$ denote the population densities of two competing species in patch $i$, respectively;
$\mu_u,\mu_v\ge0$ are the dispersal rates of the two species, respectively;
$p_i\in\mathbb{R}$ represents the intrinsic growth rates of species $u_i$ and $v_i$ in patch $i$; and
$a_{ij}\ge 0 \, (i\not=j)$ is the movement rate from patch $j$ to patch $i$, $a_{jj}=-\ds\sum_{i\neq j} a_{ij}$  is the total movement out from patch $j$, and the matrix $A=(a_{ij})$ is irreducible. Let $(\alpha_1, \alpha_2, ..., \alpha_n)^T$ be the positive eigenvector of $A$ satisfying $\ds\sum_{i=1}^n \alpha_i=1$.
The two competing species are assumed to be identical except for the dispersal rates.

Denote $M=\ds\max_{1\le i\le n}\{p_i\}$ and $m=\ds\sum_{i=1}^n\alpha_i p_i$. If $M<0$, then the trivial equilibrium is the only nonnegative equilibrium, which is globally asymptotically stable. Therefore,  in the following we assume that
$M>0$.
By Theorem \ref{thm-single}, we obtain the following result about the existence/nonexistence of nonnegative semi-trivial equilibria:
\begin{lemma}\label{semitrival}
Suppose  $M>0$ and $\mu_u<\mu_v$. Then the following results hold:
\begin{enumerate}
    \item [{\rm (i)}] if $m>0$, then system \eqref{competition} admits exactly two nonnegative semi-trivial equilibria
    $(u^*,0)$ and $(0,v^*)$, where $w^*=(w^*_1,\dots,w^*_n)$ for $w=u,v$.
    \item [{\rm (ii)}] if $m<0$, then there exists a unique $\mu_*>0$ such that $s(\mu_* A+diag(p_i))=0$. Moreover,
         when $\mu_*\le \mu_u<\mu_v$,  there exists no nonnegative semi-trivial equilibrium;
       when $\mu_u<\mu_*\le \mu_v$,  there exists exactly one nonnegative semi-trivial equilibrium
        $(u^*, 0)$;
        and when $\mu_u<\mu_v<\mu_*$,  there exist exactly two nonnegative semi-trivial equilibria
    $(u^*, 0)$ and $(0,v^*)$.
\end{enumerate}
\end{lemma}

Next we show that system \eqref{competition} has no positive equilibrium.
\begin{lemma}\label{nonexistencep}
Suppose  $M>0$ and $\mu_u<\mu_v$. If $(p_1, p_2, ..., p_n)$ is not a multiple of $(\alpha_1,\alpha_2, ..., \alpha_n)$, then system \eqref{competition} admits no nonnegative equilibrium
$(\bar u,\bar v)=(\bar u_1,\dots,\bar u_n,$
$\bar v_1,\dots,\bar v_n)$ with $\bar u\ge (\not\equiv)\, 0$ and $\bar v\ge (\not\equiv)\, 0$.
\end{lemma}
\begin{proof}
Assume on the contrary that such an equilibrium $(\bar u,\bar v)=(\bar u_1,\dots,\bar u_n,\bar v_1,\dots,\bar v_n)$ exists.
Let $Q=diag(p_i-\bar u_i-\bar v_i)$. Since $(\bar u,\bar v)$ is an equilibrium of \eqref{competition}, we have
\begin{equation*}
\begin{split}
&\bar u_i(p_i-\bar u_i-\bar v_i)+\mu_u\ds\sum_{j=1}^{n}a_{i j}\bar u_j=0, \ \  i=1,2,\dots,n,
\\
&\bar v_i(p_i-\bar u_i-\bar v_i)+\mu_v\ds\sum_{j=1}^n a _{i j}\bar v_j=0, \ \  i=1,2,\dots,n.
\end{split}
\end{equation*}
Therefore, $\bar u$ and $\bar v$ are  nonnegative eigenvectors of $\mu_u A+Q$ and $\mu_v A+Q$ corresponding with eigenvalue $0$, respectively. Since $A$ is irreducible, by the Perron-Frobenius theorem, we have
$$
s(\mu_u A+Q)=s(\mu_v A+Q)=0.
$$
Since $\mu_u<\mu_v$, by Theorem \ref{theorem_quasi0} or Theorem \ref{theorem_quasi}, $Q$ is a multiple of $I$ and $\bar u, \bar v$ are eigenvectors of $A$. It follows that $p_1-u_1-v_1=p_2-u_2-v_2=\cdots=p_n-u_n-v_n=0$ and $(p_1, p_2, ..., p_n)$ is a multiple of $(\alpha_1,\alpha_2, ..., \alpha_n)$, which is a contradiction. This completes the proof.
\end{proof}

In the following, we will use monotone dynamical system theory to investigate the global dynamics of \eqref{competition}. Let $\le_K, <_K, \ll_K$ be the order of $\mathbb{R}^n\times\mathbb{R}^n$ generated by the cone $\mathbb{R}^n_+\times (-\mathbb{R}^n_+)$ defined in the usual way. For example, $(u, v)<_K (w, z)$ means
$$
u\le w, \ \ v\ge z \ \text{and}\ \ (u, v)\neq  (w, z).
$$
Then the solutions of system \eqref{competition} induce a strictly monotone semiflow on  $\mathbb{R}^n_+\times \mathbb{R}^n_+$:
\begin{lemma}
Let $\left(u^{(i)}(t),v^{(i)}(t)\right)$  be the corresponding solutions
of \eqref{competition} with nonnegative initial value $\left(u^{(i)}_0,v^{(i)}_0\right)$ for $i=1,2$, where $u^{(2)}_0\ge (\not\equiv)\, 0$, $v^{(1)}_0\ge (\not\equiv)\, 0$ and
$$
\left(u^{(2)}_0,v^{(2)}_0\right)<_K\left(u^{(1)}_0,v^{(1)}_0\right).
$$
Then $\left(u^{(2)}(t),v^{(2)}(t)\right)\ll_K\left(u^{(1)}(t),v^{(1)}(t)\right)$ for any $t>0$.
\end{lemma}
\begin{proof}
Since $u^{(2)}(t)$ is the solution of
\begin{equation*}
\begin{cases}
u_i'=u_i(p_i-u_i-v_i)+\mu_u\ds\sum_{j=1}^{n}(a_{ij}u_j-a_{ji}u_i), &i=1,2,\dots,n,\\
u(0)=u_0^{(2)}\ge (\not\equiv)\, 0,
\end{cases}
\end{equation*}
and $A$ is quasi-positive and irreducible, we have $u^{(2)}(t)\gg0$  for all $t>0$ (see \cite{smith2008monotone}).  Similarly, $v^{(1)}(t)\gg0$  for all $t>0$.

Let $\overline u(t)= {u}^{(1)}(t)-{u}^{(2)}(t)$, ${\overline v}(t)= {v}^{(2)}(t)-{v}^{(1)}(t)$, ${\overline u}_0= {u}^{(1)}_0-{u}^{(2)}_0$ and ${\overline v}_0= {v}^{(2)}_0-{v}^{(1)}_0$.
Then $({\overline u}(t),{\overline v}(t))$ satisfies
\begin{equation}\label{com}
\begin{cases}
\ds\overline u_i=\mu_u\sum_{j=1}^n a_{ij}\overline u_j+\overline u_i\left(p_i-u_i^{(1)}-u_i^{(2)}-v_i^{(1)}\right)+u_i^{(2)}\overline v_i,\\
\ds\overline v_i=\mu_v\sum_{j=1}^na_{ij}\overline v_j+\overline v_i\left(p_i-v_i^{(1)}-v_i^{(2)}-u_i^{(2)}\right)+v_i^{(1)}\overline u_i,\\
({\overline u}(0), {\overline v}(0))\ge(\not\equiv)\,{0}.
\end{cases}
\end{equation}
Since $u_i^{(2)}, v_i^{(1)}>0$ for all $t>0$, $i=1, 2, ..., n$ and $A$ is quasi-positive and irreducible, \eqref{com} is cooperative and irreducible \cite{smith2008monotone}. It then follows  that
$\overline u_i(t),\overline v_i(t)>0$ for any $i=1,2,\dots,n$ and $t>0$. This proves the claim.
\end{proof}

Since the solutions of system \eqref{competition} induce a strictly monotone semiflow on  $\mathbb{R}^n_+\times \mathbb{R}^n_+$, we can use the theory of monotone dynamical systems in \cite{Hess, hsu1996competitive, lam2016remark, smith2008monotone}  to investigate the asymptotic behavior of \eqref{competition}. Specifically, if $({u^*}, 0)$ is the only semi-trivial equilibrium which is locally asymptotically stable, then it is globally asymptotically stable; if both $({u^*}, 0)$ and $(0,{v^*})$ exist with $({u^*}, 0)$ stable and $(0, {v^*})$ unstable and there exists no positive equilibrium, then $({u^*}, 0)$ is globally asymptotically stable. Then the following result follows from  Lemmas \ref{semitrival}-\ref{nonexistencep}:
\begin{theorem}\label{thm-competition}
Suppose that $M>0$, $\mu_u<\mu_v$, and $(p_1, p_2, ..., p_n)$ is not a multiple of $(\alpha_1,\alpha_2, ..., \alpha_n)$. Let $\mu_*$, ${u^*}$ be defined as in Lemma \ref{semitrival}.
Then the following statements hold:
\begin{enumerate}
    \item [{\rm (i)}] if $m>0$, then  semi-trivial equilibrium
    $({u^*}, 0)$ is globally asymptotically stable.
    \item [{\rm (ii)}] if $m<0$, then
    the trivial equilibrium is globally asymptotically stable for $\mu_u\ge \mu_*$; and the semi-trivial $({u^*}, 0)$  is globally asymptotically stable for $\mu_u<\mu_*$.
\end{enumerate}
\end{theorem}

\begin{proof}
By Lemma  \ref{semitrival},  $({u^*},0)$ always exists. We show that  $({u^*}, 0)$ is locally asymptotically stable whenever it exists. Linearizing \eqref{competition} at $({u^*}, 0)$, we obtain the following eigenvalue problem
\begin{equation}\label{eigc}
\begin{cases}
\lambda \phi_i=\phi_i(p_i-2 u^*_i)-u_i^*\psi_i+\mu_u\ds\sum_{j=1}^{n}(a_{ij}\phi_j-a_{ji}\phi_i), &i=1,2,\dots,n,\\
\lambda \psi_i=\psi_i(p_i-u_i^*)+\mu_v\ds\sum_{j=1}^{n}(a_{ij}\psi_j-a_{ji}\psi_i),&i=1,2,\dots,n.
\end{cases}
\end{equation}
It suffices to show ${\rm Re}\lambda<0$ for any eigenvalue $\lambda$ of \eqref{eigc}. Let $(\phi, \psi)$ with $\phi=(\phi_1, \phi_2, ..., \phi_n)$ and $\psi=(\psi_1, \psi_2, ..., \psi_n)$ be an eigenvector corresponding to $\lambda$. If $\psi= 0,$ then $\lambda$ satisfies
$$
\lambda \phi_i=\phi_i(p_i-2 u^*_i)+\mu_u\ds\sum_{j=1}^{n}(a_{ij}\phi_j-a_{ji}\phi_i), \ \ i=1,2,\dots,n,
$$
i.e. $\lambda$ is an eigenvalue of $\mu_u A+\text{diag}(p_i-2u_i^*)$.
Since $A$ is quasi-positive and irreducible and ${u^*}$ satisfies
$$
0=u^*_i(p_i-u^*_i)+\mu_u\ds\sum_{j=1}^{n}(a_{ij}u^*_j-a_{ji}u^*_i), \ \ i=1,2,\dots,n,
$$
${u^*}$ is a positive eigenvector of $\mu_uA+\text{diag}(p_i-u_i^*)$ corresponding with principal eigenvalue $s(\mu_uA+\text{diag}(p_i-u_i^*))=0$.  Therefore, $s(\mu_u A+\text{diag}(p_i-2u_i^*))<s(\mu_uA+\text{diag}(p_i-u_i^*))=0$. 
It follows that
$$
{\rm Re}\lambda\le s(\mu_u A+\text{diag}(p_i-2u_i^*))<0.
$$
Therefore, we may assume $\psi\neq 0$. Then, $\lambda$ satisfies
$$
\lambda \psi_i=\psi_i(p_i-u_i^*)+\mu_v\ds\sum_{j=1}^{n}(a_{ij}\psi_j-a_{ji}\psi_i),\ \ i=1,2,\dots,n.
$$
i.e. $\lambda$ is an eigenvalue of $\mu_vA+\text{diag}(p_i-u_i^*)$.
Since  $s(\mu_u A+\text{diag}(p_i-u^*_i))=0$ and $\mu_u<\mu_v$,
 $s(\mu_v A+\text{diag}(p_i-u^*_i))<0$ by Theorems \ref{theorem_quasi0} or \ref{theorem_quasi}. Hence, ${\rm Re}\lambda<0$. This implies that $({u^*}, 0)$ is locally asymptotically stable.
 Similarly, we have  $s(\mu_v A+\text{diag}(p_i-v^*_i))>0$ and  $({0},{v^*})$ is unstable if it exists. By Lemma \ref{nonexistencep}, \eqref{competition} has no positive equilibrium. Therefore, the results follow from the theory of strictly monotone dynamical systems \cite{Hess, hsu1996competitive, lam2016remark, smith2008monotone}.
\end{proof}

\begin{remark}
\begin{enumerate}
\item For the reaction-diffusion Lotka-Volterra competition model, it was shown in \cite{dockery1998evolution}  that the species with slower diffusion rate out-competes the one with faster diffusion rate, when the two species are identical except for the diffusion rates. Theorem \ref{thm-competition}  is an analogous result for the patch model.
\item When $p=(p_1, p_2, ..., p_n)$ is a multiple of $(\alpha_1,\alpha_2, ..., \alpha_n)$, the nonexistence of positive equilibria in Lemma \ref{nonexistencep} no longer holds. Indeed it is easy to see that for any $s\in [0,1]$, $((1-s)p,sp)$ is a nonnegative equilibrium of \eqref{competition}. The fact that $(p_1, p_2, ..., p_n)$ is a multiple of $(\alpha_1,\alpha_2, ..., \alpha_n)$ implies that the movement strategy defined by $A=(a_{ij})$ is an ideal free dispersal strategy with respect to $(p_1, p_2, ..., p_n)$, and in \eqref{competition}, both species have ideal free dispersal strategies with respect to $(p_1, p_2, ..., p_n)$, hence coexistence can be achieved (see \cite{CCL2012}). Theorem \ref{thm-competition} shows that when neither species takes the ideal free strategy, the slower disperser will prevail.
\end{enumerate}
\end{remark}

\subsection{SIS epidemic model}
Finally, we consider an SIS (susceptible-infected-susceptible) epidemic model in a heterogeneous environment. Let $S_j(t)$ and $I_j(t)$ denote the number of the susceptible and infected individuals
in patch $j$ and at time $t$, respectively. The epidemic patch model proposed by \cite{Allenpatch} is the following:
\begin{equation}\label{patc1}
\begin{cases}
\ds\f{d S_j}{dt}=\mu_S\sum_{k\in\Omega} (a_{jk} S_k- a_{kj} S_j)-\ds\f{\beta_j S_j I_j}{S_j+ I_j}+\gamma_j I_j, &j\in\Omega,\\
\ds\f{d  I_j}{dt}=\mu_I\sum_{k\in\Omega}(  a_{jk} I_k- a_{kj} I_j)+\ds\f{\beta_j S_j I_j}{  S_j+ I_j}-\gamma_j  I_j,&j\in\Omega,
\end{cases}
\end{equation}
where $\Omega=\{1,2,\dots,n\}$ with $n\ge2$; $\mu_S,\mu_I>0$ are the dispersal rates of the susceptible and infected populations, respectively; $\beta_j\ge0$
denotes the rate of disease
transmission in patch $j$, and $\gamma_j>0$ represents the rate of disease recovery in patch $j$;
and $A$ is the same as in Section \ref{section5.3}. Summing up all the equations in \eqref{patc1}, one observes that  the total population remain constant:
\begin{equation}\label{total}
\sum_{j=1}^n (S_j(t)+I_j(t))=N, \ \ \text{for all} \ t\ge 0,
\end{equation}
where $N$ is the total population.

A major assumption in \cite{Allenpatch} is that the movement rates from patch $j$ to $i$ and from patch $i$ to $j$ are the same, i.e. $a_{ji}=a_{ij}$ for all $i\neq j$. Here we do not impose this assumption. The detailed discussion of model \eqref{patc1}-\eqref{total} with asymmetric $A$ will be in a forthcoming paper \cite{chen2019}. Here, we will  briefly discuss the application of Theorems \ref{theorem_quasi0} and \ref{theorem_quasi} to this model.

Model \eqref{patc1}-\eqref{total} has a unique disease free equilibrium (i.e. the disease component is zero) $(\hat S, 0)$ with $\hat S= \alpha N$, where $\alpha=(
\alpha_i)$ is the unique positive eigenvector of $A$ satisfying $\ds\sum_{i=1}^n \alpha_i=1$ corresponding with principal eigenvalue $0$. The basic reproduction number $R_0$, a threshold value for the model, is computed by the standardized process in \cite{Diekmann, Driessche}. Specifically, the new infection and transition matrices are respectively given by
\begin{equation}\label{FV}
F=\text{diag}(\beta_j),\;\;V=\mu_I A-\text{diag}(\gamma_j),
\end{equation}
where $V+F$ can be obtained by linearizing the model \eqref{patc1}-\eqref{total} around $(\hat S, 0)$. Since $V$ is strictly diagonally dominant, $-V$ is a non-singular $M$-matrix with $-V^{-1}$ being nonnegative. Therefore, the basic reproduction number
$$
R_0=r(-FV^{-1})
$$
is well-defined, which is the principal eigenvalue of $-FV^{-1}$ by the Perron-Frobenius theorem. Moreover, if we assume $\beta_j>0$ for all $j$, then $VF^{-1}=\mu_IAF^{-1}-Q$ with $Q=\text{diag}(\gamma_j/\beta_j)$. Since $AF^{-1}$ is quasi-positive and has each column summing up to zero, it follows from Theorem~\ref{theorem_quasi0} or Theorem~\ref{theorem_quasi} that $s(VF^{-1})=s(\mu_IAF^{-1}-Q)$ is strictly decreasing in $\mu_I$ provided that $Q$ is not a multiple of $I$. Therefore, we have the following result:
\begin{theorem}\label{theorem_sis}
Suppose that $\beta_j, \gamma_j>0$ for all $j$ and $A$ is an irreducible quasi-positive matrix.
Then, the following statements hold:
\begin{enumerate}
\item [{\rm (i)}] $R_0-1$ has the same sign as $s(V+F)=s\left(\mu_I A+diag(\beta_j-\gamma_j)\right)$; if $R_0<1$, then the disease-free equilibrium $(\hat S, 0)
$ of  \eqref{patc1}-\eqref{total} is locally asymptotically stable.
\item [{\rm (ii)}] $R_0$ is strictly decreasing in $\mu_I$ with
\vskip -8mm
$$
\lim_{\mu_I\rightarrow 0} R_0=\max_{1\le j\le n}\left\{\frac{\beta_j}{\gamma_j}\right\} \;\; \text{and} \;\; \lim_{\mu_I\to\infty} R_0=\ds\f{\ds\sum_{j=1}^n\alpha_j\beta_j}{\ds\sum_{j=1}^n \alpha_j\gamma_j},
$$
provided that $(\beta_j)$ is not a multiple of $(\gamma_j)$.
\end{enumerate}
\end{theorem}
\begin{proof}
(i) follows directly from the definition of $R_0$ and \cite{Diekmann, Driessche}. Since $s(VF^{-1})=s(\mu_IAF^{-1}-Q)$ is strictly decreasing in $\mu_I$, $s(-FV^{-1})$ is strictly decreasing in $\mu_I$ if $Q$ is not a multiple of $I$. Therefore, $R_0=r(-FV^{-1})=s(-FV^{-1})$ is strictly decreasing in $\mu_I$ if $(\beta_j)$ is not a multiple of $(\gamma_j)$. The limit of $R_0$ as $\mu_I\rightarrow 0$ is obvious. To see the limit of $R_0$ as $\mu_I\rightarrow \infty$, we notice that $F\alpha=(\alpha_1\beta_1, \alpha_2\beta_2, \dots, \alpha_n\beta_n)$ is a principal eigenvector of $AF^{-1}$, which can be normalized as $(\alpha_1\beta_1, \alpha_2\beta_2, \dots, \alpha_n\beta_n)/\sum_j \alpha_j\beta_j$. Since $F^{-1}$ has each column sum equaling zero, it follows from Theorem \ref{theorem_quasi} that
\vskip -2mm
$$
\lim_{\mu_I\rightarrow 0} s(\mu_IAF^{-1}-Q)= \ds\f{\ds\sum_{j=1}^n\alpha_j\gamma_j}{\ds\sum_{j=1}^n \alpha_j\beta_j}.
$$
As a consequence,
$$
\lim_{\mu_I\to\infty} R_0=-\frac{1}{\ds\lim_{\mu_I\rightarrow 0} s(\mu_IAF^{-1}-Q)}=\ds\f{\ds\sum_{j=1}^n\alpha_j\beta_j}{\ds\sum_{j=1}^n \alpha_j\gamma_j}.
$$\vskip -8mm
\end{proof}

\begin{remark}
In Theorem \ref{theorem_sis}, we have assumed that $\beta_j$ is positive for all $j$, which is not necessary. However, if we drop this assumption, the proof will be more technical. We will leave this to the forthcoming paper \cite{chen2019}.
\end{remark}

The monotonicity of $R_0$ with respect to $\mu_I$ has been addressed as an open problem in \cite{Allenpatch}. During the preparation of our current paper, we learned that this problem was independently solved in \cite{gao2019, gao2019fast}. The proof of monotonicity of $R_0$ with respect to $\mu_I$ in \cite{gao2019fast} uses Karlin's Theorem as well.

\hspace{0.5cm}
\begin{appendix}
\setcounter{equation}{0} \setcounter{theorem}{0}
\renewcommand{\thesection}{\Alph{section}}
\renewcommand{\theequation}{\arabic{equation}}
\setcounter{section}{1}
\noindent\textbf{Appendix: Notation from graph theory and Tree-Cycle idenitity}\vspace{10pt}

Let $A=(a_{ij})$ be a nonnegative $n\times n$ matrix. A {\it weighted digraph} $\mathcal{G}=\mathcal{G}_A$ associated with $A$ can be constructed as follows: $\mathcal{G}=(V, E)$ is a pair of two sets, a set $V=\{1,2,\ldots,n\}$ of vertices and a set $E$ of arcs $(i,j)$ with weight $a_{ij}$ leading from initial vertex $j$ to terminal vertex $i$. Specifically, $(i,j) \in E(\mathcal{G})$ if and only if $a_{ij}>0$.

A digraph is \textit{strongly connected} if, for any ordered pair of distinct vertices $i,j$, there exists a directed path from $i$ to $j$. A weighted digraph $\mathcal{G}_A$ is strongly connected if and only if the weight matrix $A$ is irreducible \cite{berman1979nonnegative}.

A subdigraph $\mathcal{H}$ of $\mathcal{G}$ is \textit{spanning} if $\mathcal{H}$ and $\mathcal{G}$ have the same vertex set. The \textit{weight} of a subdigraph $\mathcal{H}$ is the product of the weights of all its arcs.
A connected subdigraph $\mathcal{T}$ of $\mathcal{G}$ is a {\it rooted in-tree} if it contains no directed cycle, and there is one vertex, called the root, that is not an initial vertex of any arcs while each of the remaining vertices is an initial vertex of exactly one arc. A subdigraph $\mathcal{Q}$  of $\mathcal{G}$ is {\it unicyclic} if it is a disjoint union of rooted in-trees whose roots form
a directed cycle. Every vertex of unicyclic $\mathcal{Q}$ is an initial vertex of exactly one arc, and thus a unicyclic graph has also been called a functional digraph \cite[page 201]{harary}.

Notice that our definitions of rooted in-trees and unicyclic graphs (functional digraphs) above are different as those in \cite{LiShuai}. Specifically, rooted out-trees and contra-function digraphs (a disjoint union of rooted out-trees whose roots form a directed cycle) are considered in \cite{LiShuai}, respectively. As a consequence, a slightly different version of Tree-Cycle identity, in analog to Theorem~2.2 in \cite{LiShuai}, can be established using Kirchhoff's Matrix-Tree Theorem \cite{Moon}.

\begin{theorem}[Tree-Cycle identity]
\label{TC}Let  $\mathcal{G}_A$  be a strongly connected weighted digraph. Let $L=(\ell_{ij})$ be the Laplacian matrix of $\mathcal{G}_A$; that is, $\ell_{ij}=-a_{ij}$ for $i\neq j$ and $\ell_{ii}=\sum_{k\not = i}a_{ki}$. Let $(\alpha_1, \alpha_2, ..., \alpha_n)^T$ be a positive, normalized principal right eigenvector of~$L$.  Then the following identity holds:
\begin{equation*}
\sum_{i, j=1}\alpha_i a_{ji}F_{ji}(x_j, x_i)= \sum_{\mathcal{Q} \in \mathbb{Q}} w(\mathcal{Q}) \sum_{(s,r)\in E(\mathcal{C}_{\mathcal{Q}})} F_{sr}(x_s, x_r),
\end{equation*}
where $\mathbb{Q}$ is the set of all spanning unicycle graphs of $(\mathcal{G}, {A})$, $w(\mathcal{Q})>0$ is the weight of $\mathcal{Q}$, and $\mathcal{C}_\mathcal{Q}$ denotes the directed cycle of $\mathcal{Q}$ with arc set $E(\mathcal{C}_{\mathcal{Q}})$.
\end{theorem}

Since $\mathcal{G}_A$ is strongly connected, equivalently, $A$ is irreducible, $0$ is a simple eigenvalue of $L$.  Let $(\alpha_1, \alpha_2, ..., \alpha_n)^T$ be a positive, normalized principal right eigenvector of $L$. It follows from Kirchhoff's Matrix-Tree Theorem that $\alpha_i = \ds\frac{C_{ii}}{\sum_{k=1}^n C_{kk}}$. Here $C_{ii}$ is the cofactor of the $i$-th diagonal entry of $L$ and can also be interpreted as  $C_{ii}=\ds\sum_{\mathcal{T}\in\mathbb{T}_i} w(\mathcal{T})$ where $\mathbb{T}_i$ is the set of all spanning in-trees rooted at $i$. Therefore, each term in the product $\alpha_i a_{ji}$ corresponds to a unicylic graph that is formed by adding arc $(j,i)$ from $i$ to $j$ to a spanning in-tree rooted at $i$. So, the same argument as in the proof of \cite[Theorem~2.2]{LiShuai} can be applied to establish Theorem~\ref{TC}, and thus is omitted.

\end{appendix}


\begin{thebibliography}{10}

\bibitem{Allenpatch}
L.~J.~S. Allen, B.~M. Bolker, Y.~Lou, and A.~L. Nevai.
\newblock Asymptotic profiles of the steady states for an {$SIS$} epidemic
  patch model.
\newblock {\em SIAM J. Appl. Math.}, 67(5):1283--1309, 2007.

\bibitem{altenberg2010karlin}
L.~Altenberg.
\newblock Karlin theory on growth and mixing extended to linear differential
  equations.
\newblock {\em arXiv preprint arXiv:1006.3147}, 2010.

\bibitem{altenberg2012resolvent}
L.~Altenberg.
\newblock Resolvent positive linear operators exhibit the reduction phenomenon.
\newblock {\em Proc. Natl. Acad. Sci. USA}, 109(10):3705--3710, 2012.

\bibitem{altenberg2017unified}
L.~Altenberg, U.~Liberman, and M.~W. Feldman.
\newblock Unified reduction principle for the evolution of mutation, migration,
  and recombination.
\newblock {\em Proc. Nat. Acad. Sci. U.S.A.}, 114(12):E2392--E2400, 2017.

\bibitem{arino2009diseases}
J.~Arino.
\newblock Diseases in metapopulations.
\newblock In {\em Modeling and dynamics of infectious diseases}, volume~11 of
  {\em Ser. Contemp. Appl. Math. CAM}, pages 64--122. Higher Ed. Press,
  Beijing, 2009.

\bibitem{arino2006disease}
J.~Arino and P.~van~den Driessche.
\newblock Disease spread in metapopulations.
\newblock In {\em Nonlinear dynamics and evolution equations}, volume~48 of
  {\em Fields Inst. Commun.}, pages 1--12. Amer. Math. Soc., Providence, RI,
  2006.

\bibitem{berman1979nonnegative}
A.~Berman and R.~J. Plemmons.
\newblock {\em Nonnegative matrices in the mathematical sciences}, volume~9 of
  {\em Classics in Applied Mathematics}.
\newblock Society for Industrial and Applied Mathematics (SIAM), Philadelphia,
  PA, 1994.

\bibitem{CCDP2007}
R.~S. Cantrell, C.~Cosner, D.~L. Deangelis, and V.~Padron.
\newblock The ideal free distribution as an evolutionarily stable strategy.
\newblock {\em J. Biol. Dyn.}, 1(3):249--271, 2007.

\bibitem{CCL2012}
R.~S. Cantrell, C.~Cosner, and Y.~Lou.
\newblock Evolutionary stability of ideal free dispersal strategies in patchy
  environments.
\newblock {\em J. Math. Biol.}, 65(5):943--965, 2012.

\bibitem{cantrell2017evolution}
R.~S. Cantrell, C.~Cosner, Y.~Lou, and S.~J. Schreiber.
\newblock Evolution of natal dispersal in spatially heterogenous environments.
\newblock {\em Math. Biosci.}, 283:136--144, 2017.

\bibitem{chen2019}
S.-S. Chen, J.-P. Shi, Z.~Shuai, and Y.-X. Wu.
\newblock Asymptotic profiles of the steady states for an {SIS} epidemic patch
  model with asymmetric connectivity matrix.
\newblock Submitted.

\bibitem{cohen1981convexity}
J.~E. Cohen.
\newblock Convexity of the dominant eigenvalue of an essentially nonnegative
  matrix.
\newblock {\em Proc. Amer. Math. Soc.}, 81(4):657--658, 1981.

\bibitem{cosner1996variability}
C.~Cosner.
\newblock Variability, vagueness and comparison methods for ecological models.
\newblock {\em Bull. Math. Biol.}, 58(2):207--246, 1996.

\bibitem{dhirasakdanon2007sharp}
T.~Dhirasakdanon, H.~R. Thieme, and P.~Van Den~Driessche.
\newblock A sharp threshold for disease persistence in host metapopulations.
\newblock {\em J. Biol. Dyn.}, 1(4):363--378, 2007.

\bibitem{Diekmann}
O.~Diekmann, J.~A.~P. Heesterbeek, and J.~A.~J. Metz.
\newblock On the definition and the computation of the basic reproduction ratio
  {$R_0$} in models for infectious diseases in heterogeneous populations.
\newblock {\em J. Math. Biol.}, 28(4):365--382, 1990.

\bibitem{dockery1998evolution}
J.~Dockery, V.~Hutson, K.~Mischaikow, and M.~Pernarowski.
\newblock The evolution of slow dispersal rates: a reaction diffusion model.
\newblock {\em J. Math. Biol.}, 37(1):61--83, 1998.

\bibitem{donsker1975variational}
M.~D. Donsker and S.~R.~S. Varadhan.
\newblock On a variational formula for the principal eigenvalue for operators
  with maximum principle.
\newblock {\em Proc. Nat. Acad. Sci. U.S.A.}, 72:780--783, 1975.

\bibitem{Du2014}
P.~Du and M.~Y. Li.
\newblock Impact of network connectivity on the synchronization and global
  dynamics of coupled systems of differential equations.
\newblock {\em Phys. D}, 286/287:32--42, 2014.

\bibitem{feldman1980evolution}
M.~W. Feldman, F.~B. Christiansen, and L.~D. Brooks.
\newblock Evolution of recombination in a constant environment.
\newblock {\em Proc. Nat. Acad. Sci. U.S.A.}, 77(8, part 2):4838--4841, 1980.

\bibitem{Freedman1989}
H.~I. Freedman and Y.~Takeuchi.
\newblock Global stability and predator dynamics in a model of prey dispersal
  in a patchy environment.
\newblock {\em Nonlinear Anal.}, 13(8):993--1002, 1989.

\bibitem{friedland1981convex}
S.~Friedland.
\newblock Convex spectral functions.
\newblock {\em Linear and Multilinear Algebra}, 9(4):299--316, 1980/81.

\bibitem{gao2019}
D.-Z. Gao.
\newblock Travel frequency and infectious diseases.
\newblock {\em SIAM J. Appl. Math.}, 79(4):1581--1606, 2019.

\bibitem{gao2019fast}
D.-Z. Gao and C.-P. Dong.
\newblock Fast diffusion inhibits disease outbreaks.
\newblock {\em arXiv preprint arXiv:1907.12229}.

\bibitem{gao2012}
D.-Z. Gao and S.-G. Ruan.
\newblock A multipatch {M}alaria model with logistic growth populations.
\newblock {\em SIAM J. Appl. Math.}, 72(3):819--841, 2012.

\bibitem{guo2006global}
H.~Guo, M.~Y. Li, and Z.~Shuai.
\newblock Global stability of the endemic equilibrium of multigroup {SIR}
  epidemic models.
\newblock {\em Can. Appl. Math. Q.}, 14(3):259--284, 2006.

\bibitem{guo2008graph}
H.~Guo, M.~Y. Li, and Z.~Shuai.
\newblock A graph-theoretic approach to the method of global {L}yapunov
  functions.
\newblock {\em Proc. Amer. Math. Soc.}, 136(8):2793--2802, 2008.

\bibitem{harary}
F.~Harary.
\newblock {\em Graph Theory}.
\newblock Addison-Wesley Publishing Co., Reading, 1969.

\bibitem{hastings1983can}
A.~Hastings.
\newblock Can spatial variation alone lead to selection for dispersal?
\newblock {\em Theoret. Population Biol.}, 24(3):244--251, 1983.

\bibitem{Hess}
P.~Hess.
\newblock {\em Periodic-Parabolic Boundary Value Problems and Positivity},
  volume 247 of {\em Pitman Research Notes in Mathematics Series}.
\newblock Longman Scientific \& Technical, Harlow, 1991.

\bibitem{holland2008strong}
M.~D. Holland and A.~Hastings.
\newblock Strong effect of dispersal network structure on ecological dynamics.
\newblock {\em Nature}, 456(7223):792--795, 2008.

\bibitem{hsu1996competitive}
S.~B. Hsu, H.~L. Smith, and P.~Waltman.
\newblock Competitive exclusion and coexistence for competitive systems on
  ordered {B}anach spaces.
\newblock {\em Trans. Amer. Math. Soc.}, 348(10):4083--4094, 1996.

\bibitem{karlin1982classifications}
S.~Karlin.
\newblock Classifications of selection-migration structures and conditions for
  a protected polymorphism.
\newblock In {\em Evolutionary Biology}, volume~14, pages 61--204. Plenum
  Press, New York, 1982.

\bibitem{kato1982superconvexity}
T.~Kato.
\newblock Superconvexity of the spectral radius, and convexity of the spectral
  bound and the type.
\newblock {\em Math. Z.}, 180(2):265--273, 1982.

\bibitem{kirkland2006evolution}
S.~Kirkland, C.-K. Li, and S.~J. Schreiber.
\newblock On the evolution of dispersal in patchy landscapes.
\newblock {\em SIAM J. Appl. Math.}, 66(4):1366--1382, 2006.

\bibitem{lam2016remark}
K.-Y. Lam and D.~Munther.
\newblock A remark on the global dynamics of competitive systems on ordered
  {B}anach spaces.
\newblock {\em Proc. Amer. Math. Soc.}, 144(3):1153--1159, 2016.

\bibitem{LiPeng}
H.-C. Li and R.~Peng.
\newblock Dynamics and asymptotic profiles of endemic equilibrium for {SIS}
  epidemic patch models.
\newblock {\em J. Math. Biol.}, 79(4):1279--1317, 2019.

\bibitem{LiShuai}
M.~Y. Li and Z.~Shuai.
\newblock Global stability of an epidemic model in a patchy environment.
\newblock {\em Can. Appl. Math. Q.}, 17(1):175--187, 2009.

\bibitem{li2010global}
M.~Y. Li and Z.~Shuai.
\newblock Global-stability problem for coupled systems of differential
  equations on networks.
\newblock {\em J. Differential Equations}, 248(1):1--20, 2010.

\bibitem{Lu1993}
Z.~Y. Lu and Y.~Takeuchi.
\newblock Global asymptotic behavior in single-species discrete diffusion
  systems.
\newblock {\em J. Math. Biol.}, 32(1):67--77, 1993.

\bibitem{Moon}
J.~W. Moon.
\newblock {\em Counting Labelled Trees}.
\newblock Canadian Mathematical Congress, Montreal, 1970.

\bibitem{schreiber2009invasion}
S.~J. Schreiber and J.~O. Lloyd-Smith.
\newblock Invasion dynamics in spatially heterogeneous environments.
\newblock {\em The American Naturalist}, 174(4):490--505, 2009.

\bibitem{smith2008monotone}
H.~L. Smith.
\newblock {\em Monotone {D}ynamical {S}ystems: {A}n {I}ntroduction to the
  {T}heory of {C}ompetitive and {C}ooperative {S}ystems}.
\newblock American Mathematical Society, Providence, RI, 1995.

\bibitem{thieme1992convergence}
H.~R. Thieme.
\newblock Convergence results and a {P}oincar\'{e}-{B}endixson trichotomy for
  asymptotically autonomous differential equations.
\newblock {\em J. Math. Biol.}, 30(7):755--763, 1992.

\bibitem{Driessche}
P.~van~den Driessche and J.~Watmough.
\newblock Reproduction numbers and sub-threshold endemic equilibria for
  compartmental models of disease transmission.
\newblock {\em Math. Biosci.}, 180:29--48, 2002.

\bibitem{wang2004epidemic}
W.-D. Wang and X.-Q. Zhao.
\newblock An epidemic model in a patchy environment.
\newblock {\em Math. Biosci.}, 190(1):97--112, 2004.

\bibitem{zhao2003dynamical}
X.-Q. Zhao.
\newblock {\em Dynamical Systems in Population Biology}, volume~16 of {\em CMS
  Books in Mathematics}.
\newblock Springer-Verlag, New York, 2003.

\end{thebibliography}
\end{document}